\documentclass[11pt,a4paper,reqno]{amsart}

\usepackage{mathrsfs} 

\usepackage{subfig} 
\usepackage{graphicx}

\usepackage{amsmath,amssymb,graphics,epsfig,color,enumerate,psfrag, mathtools}
\usepackage{dsfont}
\usepackage{verbatim}
\newtheorem{Theorem}{Theorem}[section]

\newtheorem{Lemma}[Theorem]{Lemma}
\newtheorem{Proposition}[Theorem]{Proposition}

\theoremstyle{definition}
\newtheorem{Remark}[Theorem]{Remark}

\newtheorem{Definition}[Theorem]{Definition}

\newtheorem{Assumption}[Theorem]{Assumption}

\newcommand{\cC}{\ensuremath{\mathcal C}}

\newcommand{\cF}{\ensuremath{\mathcal F}}

\newcommand{\cN}{\ensuremath{\mathcal N}}

\newcommand{\cP}{\ensuremath{\mathcal P}}

\newcommand{\cX}{\ensuremath{\mathcal X}}

\newcommand{\bbE}{{\ensuremath{\mathbb E}} }

\newcommand{\bbG}{{\ensuremath{\mathbb G}} }

\newcommand{\bbN}{{\ensuremath{\mathbb N}} }

\newcommand{\bbP}{{\ensuremath{\mathbb P}} }

\newcommand{\bbR}{{\ensuremath{\mathbb R}} }

\newcommand{\bbT}{{\ensuremath{\mathbb T}} }

\newcommand{\bbZ}{{\ensuremath{\mathbb Z}} }

\newcommand{\var}{\operatorname{Var}}

\newcommand{\laa}{\langle \! \langle}
\newcommand{\raa}{\rangle \! \rangle}

\usepackage[colorlinks=true,linkcolor=blue,citecolor=blue,urlcolor=blue,hypertexnames=false]{hyperref}

\let\a=\alpha \let\b=\beta   \let\d=\delta  
 \let\g=\gamma       \let\l=\lambda
\let\m=\mu   \let\n=\nu   \let\o=\omega      
   \let\t=\tau   
 \let\x=\xi

\newcommand{\de}{\partial}

\author[A. Bou-Rabee]{Ahmed Bou-Rabee}
\address{Ahmed Bou-Rabee. University of Pennsylvania.}
\email{ahmedmb@sas.upenn.edu}

\author[V. Silvestri]{Vittoria Silvestri}
\address{Vittoria Silvestri. University of Rome La Sapienza.}
\email{silvestri@mat.uniroma1.it}

\author[A. Yadin]{Ariel Yadin}
\address{Ariel Yadin. Ben-Gurion University of the Negev.}
\email{yadina@bgu.ac.il}

\title{Gaussian fluctuations for Internal DLA on cylinders}

\begin{document}

\begin{abstract}
Internal DLA is a discrete random growth model describing growing clusters of particles. Its limiting shape and fluctuations are well understood when the underlying graph is the $d$-dimensional lattice or the cylinder $\bbZ_N \times \bbZ$. In the latter geometry, the average fluctuations of IDLA have been shown to converge to the GFF. In this note we generalise this result by showing that, for any vertex-transitive base graph $V_N$ satisfying an eigenvalue convergence condition, the average fluctuations of IDLA on the cylinder $V_N \times \bbZ$ are given by a GFF. On the way, we present an improved bound on the clusters' maximal fluctuations, which is of independent interest and which implies a shape theorem for IDLA on $V_N \times \bbZ$ for any vertex-transitive base graph $V_N$.
\end{abstract}

\maketitle

\section{Introduction} \label{sec:intro}
Internal Diffusion Limited Aggregation (IDLA) is a random growth model in which a cluster of particles grows one site at a time, with each new site sampled from the harmonic measure on the cluster boundary seen from an internal source. Meakin and Deutch~\cite{meakin1986formation} introduced the model, and Diaconis and Fulton~\cite{diaconis1991growth} rediscovered it a few years later. On the lattice $\bbZ^d$, the cluster starts from the single-site seed $A(0) = \{0\}$ and grows by
	\[ A(t) \coloneqq A(t-1) \cup \{ Z_t \}\,, \]
for each positive integer $t$, where $Z_t$ is the first exit location from $A(t-1)$ of an independent simple random walk started at the origin. Thus $A(t)$ is the IDLA cluster after $t$ steps.

Lawler, Bramson and Griffeath~\cite{lawler1992internal} proved a shape theorem: after rescaling, the cluster $A(t)$ converges in Hausdorff distance, with high probability, to a Euclidean ball. Lawler~\cite{lawler1995subdiffusive} then obtained the first subdiffusive bound on the maximal deviation of $A(t)$ from this limiting ball. Subsequent work of Jerison, Levine and Sheffield~\cite{jerison2012logarithmic, jerison2013internal} and of Asselah and Gaudilli\`ere~\cite{asselah2013logarithmic, asselah2013sublogarithmic, asselah2014lower} sharpened this maximal deviation to be square-root-logarithmic in the cluster radius in dimensions $d \geq 3$ (with a matching lower bound) and logarithmic in dimension $d = 2$. Sharpness of the logarithmic bound in dimension $d = 2$ remains open.

Jerison, Levine and Sheffield~\cite{jerison2014internal} also studied the average fluctuations of IDLA and showed that on $\bbZ^d$ they converge to a variant of the Gaussian Free Field (GFF). In a companion paper~\cite{jerison2014internal2}, the same authors considered IDLA on the cylinder $\bbZ_N \times \bbZ$ and proved that space-time averages of these fluctuations converge to the ordinary GFF; they also stated an analogous logarithmic maximal-fluctuation bound.

\begin{figure}[t]
	\centering
	\includegraphics[width=0.31\textwidth]{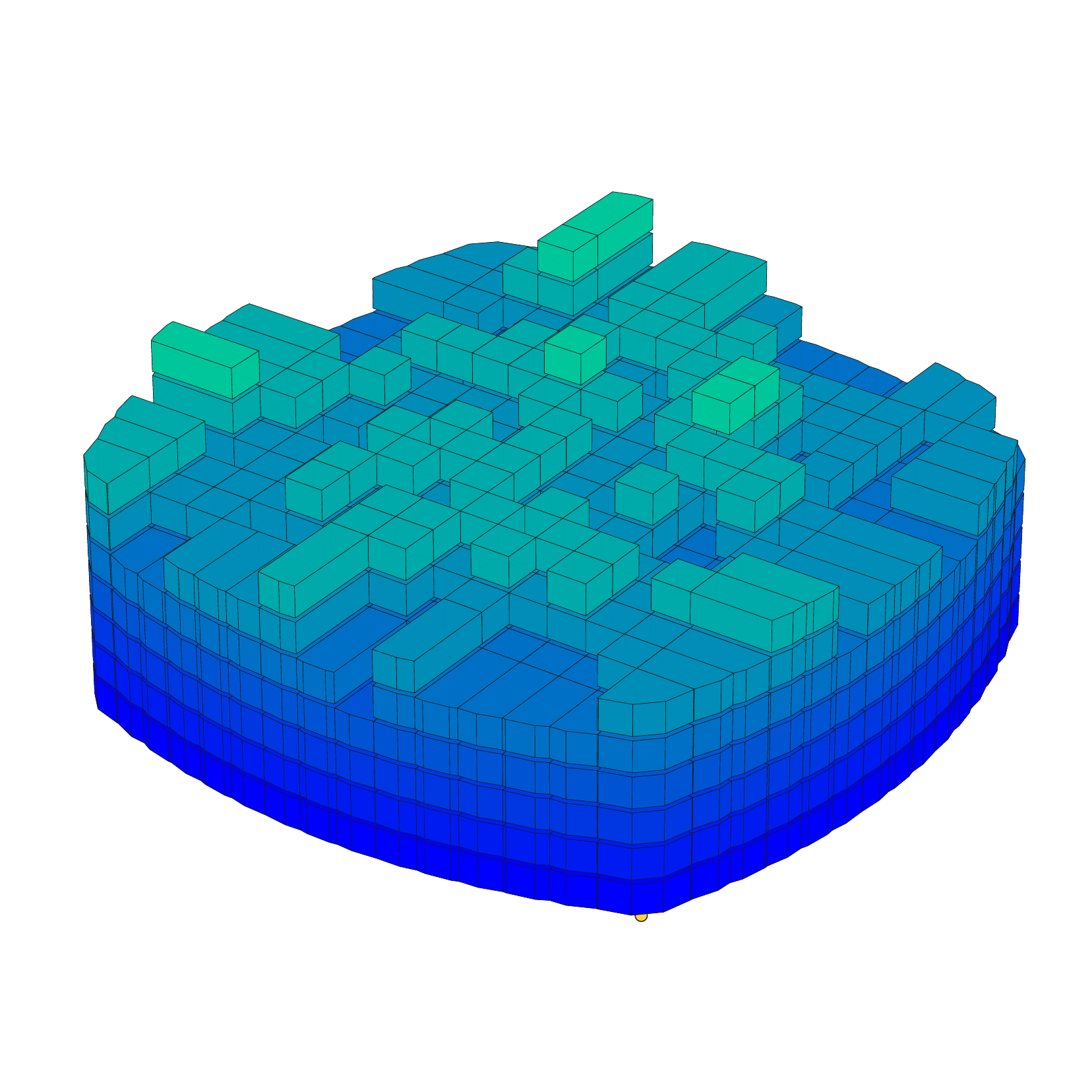}\hfill
	\includegraphics[width=0.31\textwidth]{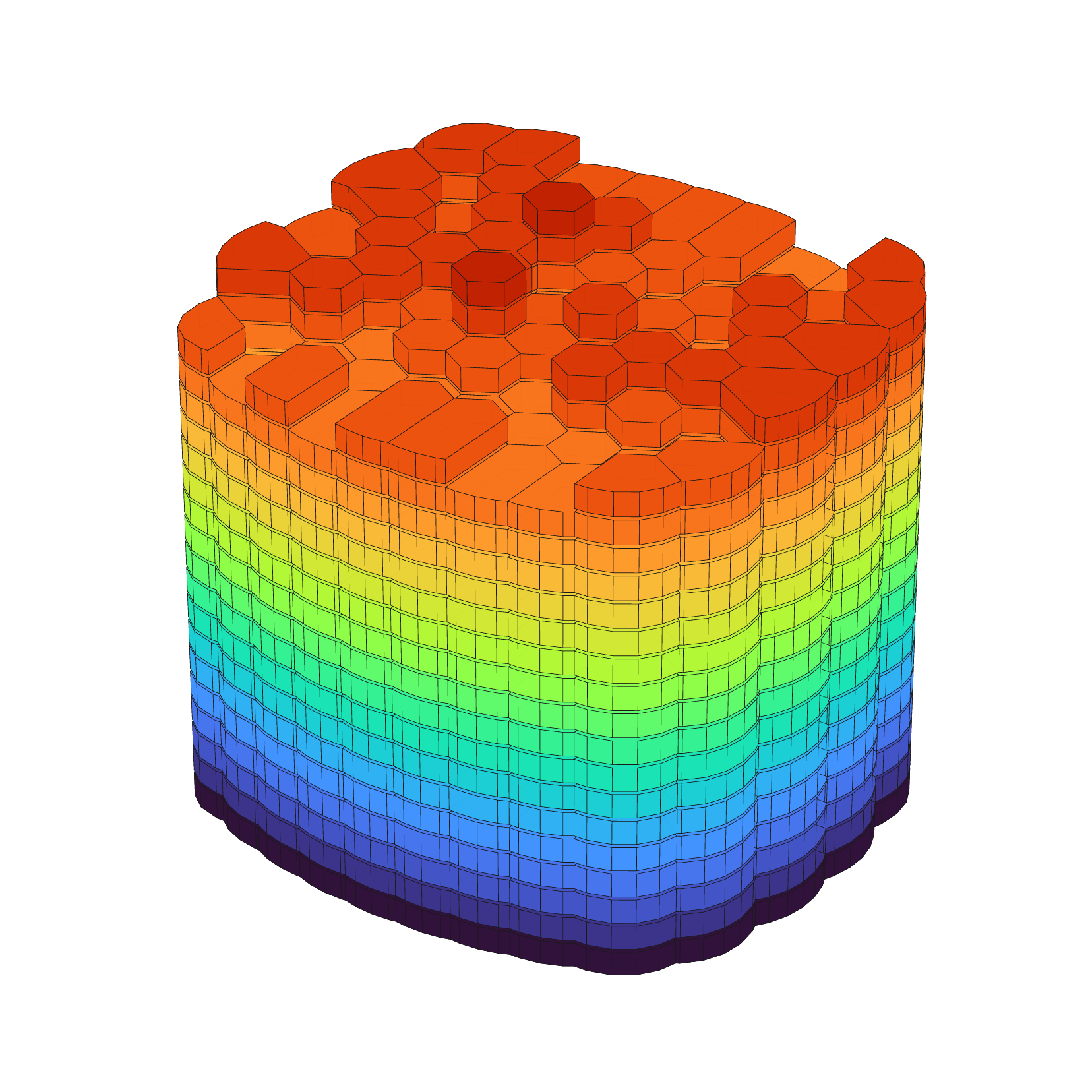}\hfill
	\includegraphics[width=0.31\textwidth]{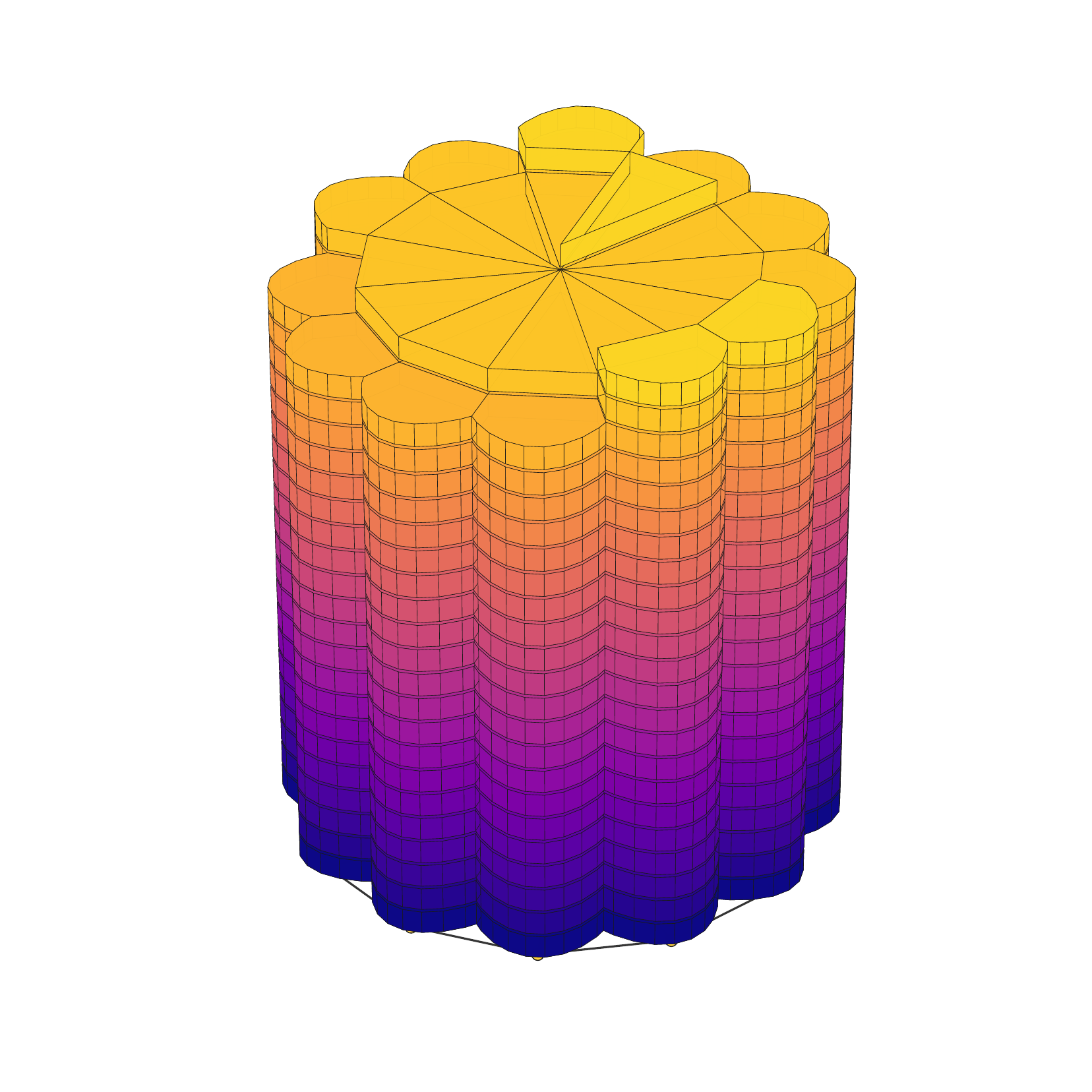}
	\caption{Three samples of the IDLA cluster $A(T)$ on the cylinder $\bbG_N = V_N \times \bbZ$ for different vertex-transitive base graphs $V_N$: the $15 \times 15$ square lattice torus ($N = 225$), the $9 \times 9$ triangular lattice torus ($N = 81$), and the Nauru graph $\mathrm{GP}(12,5)$ ($N = 24$). In each image, every site of the cluster is drawn as a prism whose cross-section is the Voronoi cell of its base vertex in the depicted embedding.}
	\label{fig:GZ}
\end{figure}

Our main result shows that the convergence to the GFF on cylinders of~\cite{jerison2014internal2} extends well beyond the cycle base. As the prototypical example, consider $V_N = \bbZ_n^d$ with $N = n^d$, equipped with the lazy simple random walk: the rescaled horizontal operator converges to $-\tfrac{1}{2d}\Delta_{\bbT^d}$ on the flat torus $\bbT^d = \bbR^d/\bbZ^d$, and we show, roughly, that the IDLA fluctuations converge to the GFF on the half-cylinder $\bbT^d \times \bbR_+$. Theorem~\ref{th:GFF} establishes this for every vertex-transitive base whose rescaled low spectrum converges and with vanishing spectral gap. Along the way we prove a strong  maximal-fluctuation bound, which in particular implies a shape theorem for all base graphs at arbitrary polynomial times (see Remark~\ref{rem:shape_theorem}). 

\subsection{Statement of results} \label{sec:statements}

For each integer $N \geq 2$, let $V_N = (V, E)$ be a finite, connected, vertex-transitive graph on $N$ vertices (so that $|V| = N$); slightly abusing notation, we also write $V_N$ for the vertex set $V$. Vertex-transitivity means that the automorphism group of $V_N$ acts transitively on the vertex set $V$; in particular, every vertex has a common degree $d = d(N) \geq 1$, allowed to depend on $N$. Write $\sim$ for the adjacency relation in $V_N$. Let $P_N$ be the transition matrix of the lazy simple random walk on $V_N$:
\begin{equation}\label{eq:PN}
P_N(x, x') \coloneqq \begin{cases}
1/(2d) & \text{if } x \sim x', \\
1/2 & \text{if } x = x', \\
0 & \text{otherwise.}
\end{cases}
\end{equation}
By symmetry of the adjacency relation, $P_N$ is symmetric, and a fortiori doubly stochastic and reversible with respect to the uniform distribution $\pi_N$ on $V_N$. Connectedness gives irreducibility, so the second-largest eigenvalue $\l_2(N)$ is strictly less than $1$; laziness gives that all eigenvalues of $P_N$ lie in $[0, 1]$.

We construct IDLA on the cylinder graph 
\[
\bbG_N \coloneqq V_N \times \bbZ = \{ (x,y) : x \in V_N,\; y \in \bbZ \}\,.
\]
We will often refer to $x$ and $y$ as the horizontal and vertical coordinate respectively. 

The cylinder walk on $\bbG_N$ is the discrete-time Markov chain with transition kernel
 	\begin{equation}\label{def:SRW}
 	 \cP_N\big( (x,y) ; (x',y') \big) =
 	\begin{cases}
 	\frac{1}{4} & \mbox{if } x=x' \mbox{ and } |y-y'|=1 \mbox{ (vertical step),}
 	\\[3pt]
 	\frac{1}{2}\, P_N(x,x') & \mbox{if } y=y' \mbox{ (horizontal step),}
 	\\[3pt]
 	0 & \mbox{otherwise,}
 	\end{cases}
 	\end{equation}
for each $(x,y)$ and $(x',y')$ in $\bbG_N$. The vertical coordinate has zero drift and makes a step with probability $1/2$, independently of the horizontal structure.

\medskip

For each real number $k$, let $R_k$ denote the infinite half-cylinder of height $k$ given by
	\[ R_k \coloneqq \{ (x,y) \in \bbG_N : y \leq k \}\,. \]
We define the IDLA process $(A(t))_{t \geq 0}$ on $\bbG_N$ starting from the flat configuration $R_0$ inductively. Set $A(0) \coloneqq R_0$. For each integer $t \geq 1$, set
	\[ A(t) \coloneqq A(t-1) \cup \{ (X_{t} , Y_{t} ) \}\,, \]
where $(X_{t} , Y_{t})$ is the exit location from $A(t-1)$ of an independent cylinder walk~\eqref{def:SRW} started at a site chosen uniformly from $V_N \times \{0\}$. Since $\pi_N$ is uniform, this procedure grows clusters according to the harmonic measure on the boundary seen from the layer $V_N \times \{0\}$. The aggregates are increasing, each contains the half-cylinder $R_0$, and setting $A_+(t) \coloneqq A(t) \cap \{ y \geq 1\}$ we have $|A_+(t)| = t$ for every $t \geq 0$, since each step aggregates exactly one site. Figure~\ref{fig:GZ} shows three examples.

The base graph being regular, it is natural to expect that $A(T)$ is close to $R_{T/N}$. 
Our first result is a bound on the maximal fluctuations of IDLA on the cylinder $\bbG_N$ from the rectangular shape $R_{T/N}$. Let $\t_{mix}$ denote the total variation mixing time of $P_N$ at threshold $1/4$, that is, the smallest integer $t \geq 0$ with $\max_{x \in V_N} \| P_N^t(x, \cdot) - \pi_N\|_{TV} \leq 1/4$.
Write $\log_+(s) \coloneqq \max(\log s, 0)$ for brevity. 
\begin{Theorem}\label{th:max_fluct}
Let $(A(t))_{t\geq 0}$ denote an IDLA process on the cylinder $\bbG_N$ starting from $A(0) = R_0$. For every $\n > 0$ and every $m \geq 1$ there exists a constant $C = C(\n, m) < \infty$ such that, with
	\[ T^\sharp \coloneqq N \sqrt{\t_{mix}} \,(\log N)^2\,, \qquad \Delta_N \coloneqq C \max\!\left\{\, \log N,\; \sqrt{\frac{T \wedge T^\sharp}{N}\,\Big( \log N + \log_+\!\big( T/N \big) \Big)} \,\right\}\,, \]
we have
	\begin{equation}\label{eq:fluct}
	\bbP \Big( R_{\lfloor T/N - \Delta_N \rfloor}
	\subseteq A(T) \subseteq R_{T/N +  \Delta_N } \Big) \geq 1-N^{-\n}
	\end{equation}
uniformly over integer $T$ with $1 \leq T \leq N^m$, for every $N$ large enough. 
\end{Theorem}

\begin{Remark}\label{rem:shape_theorem}
Theorem~\ref{th:max_fluct} proves a shape theorem for IDLA on the cylinder graph $\bbG_N$ for any vertex-transitive base graph $V_N$. Indeed, for a lazy simple random walk on $V_N$ it is well known that 
in vertex transitive graphs
$\tau_{mix} = \mathcal{O} ( N^2)$~\cite[Corollary~13.24]{levin2017markov}, which implies $T^\sharp \leq (N \log N )^2$ from which
	\[ \Delta_N  \leq C N^{1/2} ( \log N )^{3/2} = o(N)  \]
for $N$ large enough.
\end{Remark}

Theorem~\ref{th:max_fluct} improves on the fluctuation bound obtained by Silvestri in~\cite{silvestri2020internal}, who proved that a similar statement holds\footnote{For a random walk on a regular graph, $(\t_{rel} - 1) \log 2 \leq \t_{mix} \leq \t_{rel} \log(4N)$ (see~\cite{levin2017markov}, Theorems~12.3 and~12.4).} with $\Delta_N = C \sqrt{\t_{mix}} (\log N)^2$ uniformly for $T \leq N^m$ (see~\cite{silvestri2020internal}, Theorem~4). 
Silvestri's bound relies on coupling arguments which require the random walk trajectories to mix before exiting the cluster, and it therefore cannot be improved. Instead, we prove a tighter maximal fluctuations bound from scratch by combining Freedman's martingale inequality~\cite{freedman1975tail} for the inner bound with the alive/ghost decomposition of Lawler, Bramson and Griffeath~\cite{lawler1992internal} for the outer bound.

\begin{Remark}
	For clarity of exposition we restrict throughout to vertex-transitive bases with the lazy simple random walk. We expect the maximal fluctuation result, Theorem~\ref{th:max_fluct} above, to extend to any cylinder with a doubly stochastic horizontal kernel, since the shape-theorem proof rests only on uniform-measure preservation. The GFF result, Theorem~\ref{th:GFF}, on the other hand, makes essential use of the horizontal automorphism invariance of the IDLA law so the natural extension is a kernel $P_N$ that is symmetric, lazy, irreducible, and admits a kernel-preserving group of automorphisms acting transitively on $V_N$. It would also be interesting to extend to non-homogeneous discretizations of compact Riemannian manifolds. We do not pursue these generalizations here.
\end{Remark}

\medskip

Once a good maximal fluctuation bound is in place, one can study the average fluctuations of IDLA clusters around the rectangular shape. Under our standing assumption that $V_N$ is vertex-transitive and $P_N$ is the lazy simple random walk, $P_N$ has $N$ real eigenvalues
	\begin{equation}\label{eq:eigenvalues}
	 1 = \l_1(N) > \l_2(N) \geq \cdots \geq \l_N(N) \geq 0
	\end{equation}
where strict inequality $\l_2(N) < 1$ comes from irreducibility and the lower bound $\l_N(N) \geq 0$ comes from laziness. Denote the  corresponding real eigenvectors by $\{ f^N_j : 1\leq j \leq N\}$. The eigenvectors can be taken to be an orthonormal basis of $\bbR^{V_N}$ with respect to the {\em normalised} counting-measure inner product
	\begin{equation}\label{eq:product}
	 \langle f , g \rangle \coloneqq \frac 1N \sum_{x \in V_N } f(x) g(x)\,.
	 \end{equation}
The spectral gap is $1 - \l_2(N)$, and the reciprocal
	\[ \t_{rel} \coloneqq \frac{1}{1-\l_2(N)} \]
is referred to as the relaxation time.

We formulate our spectral assumption in terms of the rescaled operator
	\begin{equation}\label{def:LN}
	 L_N \coloneqq 2a_N^2 (I - P_N)\,,
	\end{equation}
where $a_N$ is a sequence of positive reals (with $a_N \to \infty$) chosen so that the second eigenvalue of $L_N$ converges to a finite positive limit, i.e.\
	\begin{equation}\label{eq:aN}
	a_N \sqrt{2(1-\l_2(N))} \to \g_2 \in (0, \infty)
	\end{equation}
as $N\to\infty$. Such an $a_N$ exists whenever the spectral gap vanishes, and is then determined up to a multiplicative constant; the value of $\gamma_2$ encodes the choice of normalization. From~\eqref{eq:aN}, $a_N \asymp \sqrt{\t_{rel}}$ as $N\to\infty$. We will impose the additional assumption $a_N \gg \log N + \log_+ a_N$.

\begin{Assumption}\label{a:spectral}
For each fixed integer $k \geq 2$, the $k$-th eigenvalue $2a_N^2(1 - \l_k(N))$ of $L_N$ converges to a finite positive limit $\g_k^2 \in (0, +\infty)$. Equivalently,
	\[ \lim_{N\to\infty} a_N \sqrt{2(1-\l_k(N))} = \g_k \in (0, +\infty)\,. \]
\end{Assumption}
By monotonicity of the gaps, the sequence $(\gamma_k )_{k\geq 2}$ is non-decreasing, with $\gamma_2 >0$ by definition of $a_N$.

\begin{Remark}
Assumption~\ref{a:spectral} is a low-spectrum convergence condition on the rescaled horizontal operator $L_N$. When $V_N$ is a discretization of a compact Riemannian manifold $M$ and $P_N$ approximates the heat semigroup on $M$ at scale $a_N$, the eigenvalues $\g_k^2$ are exactly the eigenvalues of $-\Delta_M$ (up to overall normalization), and Assumption~\ref{a:spectral} reduces to spectral convergence of graph Laplacians.
\end{Remark}

Fix an arbitrary $y_0 \in (0,\infty )$ (this will be the limiting rescaled height of the cluster, the actual lattice height being $T/N \sim y_0 a_N$), and set
	\[ T = N \lfloor y_0 a_N \rfloor\,. \]
We will show that, for each fixed finite-mode test function, the average fluctuations of $A(T)$, appropriately rescaled, converge to a centred Gaussian random variable in the horizontal coordinate.

Write $\Omega_N$ for the (countable) state space of the IDLA process on $V_N \times \bbZ$, and recall that $R_k \in \Omega_N$ for each integer $k \geq 0$. By the growth rules, $A(0) = R_0$ and $A(t) \supset R_0$ for all $t \geq 1$. Recall that we write
	\[ A_+(t) \coloneqq A(t) \cap \{ y\geq 1\} \]
for the part of $A(t)$ strictly above level zero, so that $| A_+ (t) | = t$ for all $t\geq 0$. 
For each $A \in \Omega_N$ define the indicator function
$\mathbf{1}_A : V_N \times \bbZ \to \{ 0,1\} $ by setting 
	\[ \mathbf{1}_A(x,y) = \begin{cases}
	1 & \mbox{if } (x,y) \in A, \\
	0 & \mbox{otherwise.}
	\end{cases} \]

\begin{Definition}[Discrepancy function]
For $t \geq 0$, the discrepancy function $D_t : V_N \times \bbZ \to \{-1, 0, +1\}$ is defined by
	\[ D_t(x, y) \coloneqq \mathsf{1}_{A(t)}(x, y) - \mathsf{1}_{R_{t/N}}(x, y)\,. \]
\end{Definition}
By Theorem~\ref{th:max_fluct}, $A(T) \approx R_{T/N}$ for large $N$, so $D_T$ measures, with sign, the deviation of the cluster from the limiting rectangular shape. The symmetric difference $A(t)\triangle R_{t/N}$ is always finite, so the natural pairing
	\[ \laa D_t, f \raa \coloneqq \sum_{(x,y) \in V_N \times \bbZ} D_t(x, y)\, f(x, y) \]
makes sense as an absolutely convergent finite sum for every function $f : V_N \times \bbZ \to \bbR$. In the proofs we frequently use this pairing with test functions $f$ that satisfy the following property.

\begin{Definition}\label{def:zero}
A function $f : V_N \times \bbZ \to \bbR$ has the zero horizontal average property if
	\begin{equation}\label{eq:zero_average}
	\sum_{x \in V_N} f(x, y) = 0 \qquad \text{for every } y \in \bbZ\,.
	\end{equation}
\end{Definition}

When $f$ has the zero horizontal average property, every complete horizontal layer contributes zero to $\laa D_t, f\raa$, and the pairing reduces to a (signed) sum over the (finitely many) layers in the symmetric difference of $A(t)$ and $R_{t/N}$:
	\[ \laa D_t, f\raa = \sum_{(x,y) \in A_+(t)} f(x,y) 
= \sum_{(x,y) \in A_+(t) \setminus R_{t/N} } f(x,y) - \sum_{(x,y) \in R_{t/N} \setminus A_+(t)} f(x,y) 
	. \]

\begin{Remark}
Note that if $f$ is of the form 
	\[ f (x,y) = \sum_{k=2}^N \alpha_k (y) f_k^N(x) \] 
then it satisfies the zero horizontal average property \eqref{eq:zero_average}, 
because of the orthogonality of $(f^N_k)_k$ and the fact that $f_1^N \equiv 1 $ is the constant function.
\end{Remark}

Recall that, under Assumption~\ref{a:spectral}, the rescaled gap limits $\g_k = \lim a_N\sqrt{2(1-\l_k(N))}$ are finite and $0 < \g_2 \leq \g_3 \leq \cdots$. 
Our main result is the following. 
\begin{Theorem}\label{th:GFF}
Fix $y_0 \in (0, \infty)$. Let $\phi = \phi_N : V_N \times \bbR \to \bbR$ be a test function of the form
	\[ \phi (x,y) = \sum_{k=2}^K \alpha_k(y) f_k^N(x) \]
for some finite integer $K \geq 2$ independent of $N$, where each coefficient $\alpha_k$ does not depend on $N$ and is differentiable on a neighbourhood of $y_0$ with bounded derivative there.
Let $\varphi = \varphi_N : V_N \times \bbZ \to \bbR$ be the extension
	\[ \varphi (x,y) \coloneqq \phi \Big( x , \frac{y}{a_N} \Big)\,, \qquad (x,y) \in V_N \times \bbZ\,, \]
with $a_N$ as in~\eqref{eq:aN}. Assume that Assumption~\ref{a:spectral} holds for each $2 \leq k \leq K$, and that
	\[ a_N \gg \log N + \log_+ a_N \qquad \text{as } N \to \infty\,. \]
Let
	\[ \sigma^2(\phi) \coloneqq \sum_{k=2}^K \alpha_k(y_0)^2 \, \frac{1 - e^{-2\g_k y_0}}{2\g_k} \,\geq\, 0\,. \]
Then, with $T = N \lfloor y_0 a_N \rfloor$,
	\begin{equation}\label{eq:thGFF}
	 \frac{1}{\sqrt{Na_N}} \laa D_T , \varphi \raa \Rightarrow
	\mathcal{N} \big( 0\,,\, \sigma^2(\phi) \big)
	\end{equation}
as $N\to\infty$. 

Here $\Rightarrow$ denotes convergence in distribution, and $\mathcal{N} (0, \sigma^2)$ denotes a normal distribution with $0$ mean and variance $\sigma^2$.
\end{Theorem}

When the base discretizations converge spectrally to a compact manifold $M$, the variance formula in Theorem~\ref{th:GFF} takes a clean continuum form; this connection is discussed in the next subsection.

\subsection{Continuum interpretation: half-cylinder GFF and the fractional Gaussian field}\label{sec:FGF}

When $V_N = \bbZ_n^d$, the scalar CLT of Theorem~\ref{th:GFF} has variance matching the slice at height $y_0$ of the GFF on the half-cylinder $\bbT^d \times \bbR_+$ (modulo a constant rescaling, discussed below, that reflects the laziness of the kernel). The same mode-by-mode computation extends to a general compact Riemannian manifold $M$. As $y_0 \to \infty$, the slice covariance converges to one half of the covariance of the fractional Gaussian field (FGF) of order $s = 1/2$ on the base; we refer to~\cite{lodhia2014fractional} for a survey on the FGF.

Recall that $I - P_N$ has eigenvalues $\n_k(N) = 1 - \l_k(N)$ and the same orthonormal eigenfunctions $\{f_k^N\}$.
For $s > 0$, the FGF of order $s$ on $V_N$ is the centered Gaussian field on the mean-zero subspace $L^2_0(V_N) \coloneqq \{f \in \bbR^{V_N} : \sum_{x} f(x) = 0\}$ with covariance operator given by the inverse of $(I - P_N)^s$ restricted to $L^2_0(V_N)$: for each zero-mean test function $\phi = \sum_{k=2}^K \widehat\phi_k \, f_k^N$ on $V_N$,
	\begin{equation}\label{eq:FGFvar}
	 \bbE \big[ \langle h_N , \phi \rangle^2 \big]
	= \sum_{k=2}^K \widehat\phi_k^2 \, \n_k(N)^{-s}\,.
	\end{equation}
The case $s = 1$ is the discrete Gaussian free field (GFF), and the variance in eigenmode $k$ is an inverse power $\n_k(N)^{-s}$ of the Laplacian eigenvalue.

Consider a compact Riemannian manifold $M$ with Laplace--Beltrami operator $-\Delta_M$, and let $\kappa > 0$ be a positive constant. Write $0 < \m_2 \leq \m_3 \leq \cdots$ for the eigenvalues of $\kappa(-\Delta_M)$ restricted to $L^2_0(M)$, and let $\{\phi_k\}_{k\geq 2}$ be the corresponding eigenfunctions. After projection onto the mean-zero horizontal sector, the GFF on the half-cylinder $M \times \bbR_+$ with Dirichlet boundary at $y = 0$ and horizontal generator $\kappa(-\Delta_M)$ decomposes as $h_0(x, y) = \sum_{k\geq 2} \a_k(y)\phi_k(x)$, where the Dirichlet energy separates into independent modes:
	\[ \|h\|_\nabla^2 = \sum_{k\geq 2} \int_0^\infty \big[ |\a_k'(y)|^2 + \m_k |\a_k(y)|^2 \big]\, dy\,. \]
Each $\a_k$ is therefore an Ornstein--Uhlenbeck process with mean-reversion rate $\sqrt{\m_k}$, started from zero (see~\cite{sheffield2007gaussian} and~\cite{jerison2014internal2}). Its variance at height $y_0$ is the Green's function $G_k(y_0, y_0)$ for the operator $-\de_y^2 + \m_k$ on $(0, \infty)$ with Dirichlet boundary at $y = 0$. The method of images gives
	\[ G_k(y, y') = \frac{1}{2\sqrt{\m_k}} \Big( e^{-\sqrt{\m_k}\, |y - y'|} - e^{-\sqrt{\m_k}(y + y')} \Big)\,, \]
and setting $y = y' = y_0$ yields
	\begin{equation}\label{eq:GFFrestriction}
	 \var(\a_k(y_0)) = \frac{1 - e^{-2\sqrt{\m_k}\, y_0}}{2\sqrt{\m_k}}\,.
	\end{equation}
At finite $y_0$, the variance~\eqref{eq:GFFrestriction} is the slice covariance of the half-cylinder GFF; as $y_0 \to \infty$ the exponential correction $e^{-2\sqrt{\m_k}\,y_0}$ vanishes and the variance approaches $\tfrac{1}{2}\m_k^{-1/2}$, which is $1/2$ times the per-mode variance of the order-$1/2$ FGF on $M$ with generator $\kappa(-\Delta_M)$.  The exponential correction $e^{-2\sqrt{\m_k}\,y_0}$ encodes the memory of the flat initial configuration at $y = 0$.

Under Assumption~\ref{a:spectral}, the rescaled spectral gaps satisfy $a_N\sqrt{2\n_k(N)} \to \g_k$ as $N \to \infty$, where $\n_k(N) = 1 - \l_k(N)$. The factor of $2$ in the scaling reflects the cylinder kernel~\eqref{def:SRW}, which makes a horizontal step with probability $1/2$ rather than $1$: the effective horizontal generator at scale $a_N$ is $L_N = 2a_N^2(I - P_N)$. On a discretization of a compact Riemannian manifold $M$, the eigenvalues of $L_N$ converge to those of $\kappa(-\Delta_M)$ for a proportionality constant $\kappa$ determined by the geometry and normalization. Writing Assumption~\ref{a:spectral} in the manifold form
	\begin{equation}\label{eq:specconv}
	 2a_N^2(1 - \l_k(N)) \to \m_k \qquad \text{for each fixed } 2 \leq k \leq K\,,
	\end{equation}
we identify $\m_k = \g_k^2$ with the $k$-th eigenvalue of $\kappa(-\Delta_M)$, so that the variance factor $(1 - e^{-2\g_k y_0})/(2\g_k)$ in~\eqref{eq:thGFF} matches~\eqref{eq:GFFrestriction} exactly. For instance, when $V_N = \bbZ_n^d$ with lazy simple random walk and $a_N = n$, a direct computation gives $\kappa = 1/(2d)$, so the limit operator is $-(1/(2d))\Delta_{\bbT^d}$ in the standard flat metric on $\bbT^d = \bbR^d/\bbZ^d$; the cycle case $d = 1$ gives $\kappa = 1/2$. When $V_N = \bbZ_N$ and $a_N = N$, we recover the finite-dimensional content of the result of Jerison, Levine and Sheffield~\cite{jerison2014internal2}.

For the manifold identification we assume a consistent choice of orthonormal basis $\{f_k^N\}$ inside each multiplicity block, identifying $f_k^N$ with the corresponding continuum eigenfunction $\phi_k$ of $\kappa(-\Delta_M)$. (For the torus examples below the Fourier basis gives a canonical such choice.) Applying Theorem~\ref{th:GFF} to the test function $\varphi_g(x, y) \coloneqq \sum_{k=2}^K c_k f_k^N(x)$, independent of $y$ so that $\phi_g = \varphi_g$, with $g = \sum_{k=2}^K c_k \phi_k$ a finite linear combination of mean-zero Laplace eigenfunctions on $M$, yields
	\begin{equation}\label{eq:cormanifold}
	\frac{1}{\sqrt{Na_N}} \laa D_T, \varphi_g \raa
	\;\Rightarrow\; \cN\!\left( 0\,, \sum_{k=2}^K c_k^2\, \frac{1 - e^{-2\sqrt{\m_k}\, y_0}}{2\sqrt{\m_k}} \right)\,,
	\end{equation}
which is the variance of $\langle g, h_0(\cdot, y_0)\rangle$ for the half-cylinder GFF $h_0$ on $M \times \bbR_+$ with horizontal generator $\kappa(-\Delta_M)$.

The convergence can be upgraded to a field-valued one if a uniform second-moment hypothesis on the full spectrum is satisfied. This hypothesis can be checked on cycles and tori by a direct Fourier computation. Specifically, define $X_N(g) \coloneqq (Na_N)^{-1/2}\laa D_T, \varphi_g\raa$ for each finite-mode mean-zero test $g$. If the second-moment bound
	\begin{equation}\label{eq:uniform-H}
	\tag{H}
	\bbE\!\big[ X_N(g)^2 \big] \;\leq\; C \sum_{k=2}^K c_k^2\,\m_k^{-1/2}
	\end{equation}
holds uniformly in $N$, with implicit constant independent of the finite mode support of $g$, then a standard argument identifies the limit in $H^{-s}_0(M)$ for every $s > (d-1)/2$ (where $d = \dim M$) as the centered Gaussian field with covariance
	\begin{equation}\label{eq:operatorvar}
	 \big\langle g\,,\, \cC_{y_0}\, g \big\rangle_{L^2(M)}\,, \qquad
	 \cC_{y_0} \coloneqq \Pi_0\,\frac{1 - e^{-2y_0\sqrt{\kappa(-\Delta_M)}}}{2\sqrt{\kappa(-\Delta_M)}}\,\Pi_0\,,
	\end{equation}
where $\Pi_0$ is the orthogonal projection onto the mean-zero subspace $L^2_0(M)$. For $d = 1$, this argument recovers the field-level convergence result of~\cite{jerison2014internal2}; for $d \geq 2$, the threshold $s > (d-1)/2$ is the one given by the half-Laplacian Gaussian field's Sobolev embedding on $M$.
\subsection{Examples}
We illustrate Theorem~\ref{th:GFF} on the cycle $\bbZ_N$ and the tori $\bbZ_n^d$ for $d \geq 1$. For $d = 1$ we recover the finite-dimensional content of the result of Jerison, Levine and Sheffield~\cite{jerison2014internal2}; for $d \geq 2$ the result is new, and identifies the scalar limits with the slice at height $y_0$ of the half-cylinder GFF on $\bbT^d \times \bbR_+$ with horizontal generator $-(1/(2d))\Delta_{\bbT^d}$.

\subsubsection{The cycle base $\bbZ_N$}
For $N \geq 2$, let $V_N = \bbZ_N$ be the cycle on $N$ vertices. With our convention $f_1^N \equiv 1$, a complete real orthonormal basis of eigenfunctions of $P_N$ is given by
\begin{equation}\label{eq:fj}
	f_1^N \equiv 1\,, \qquad f_{2j}^N(x) = \sqrt{2} \cos\!\Big(\frac{2\pi j x}{N}\Big)\,, \qquad f_{2j+1}^N(x) = \sqrt{2} \sin\!\Big(\frac{2\pi j x}{N}\Big)
\end{equation}
for $j = 1, 2, \ldots, \lfloor (N-1)/2\rfloor$ and $x \in \bbZ_N$, plus the additional eigenfunction $f_N^N(x) = (-1)^x$ with eigenvalue $\l_N(N) = 0$ when $N$ is even. The corresponding eigenvalues are $\l_1(N) = 1$ and
	\[ \l_{2j}(N) = \l_{2j+1}(N) = \frac{1}{2} + \frac{1}{2}\cos\!\Big(\frac{2\pi j}{N}\Big) = 1 - \frac{\pi^2 j^2}{N^2} + O(j^4/N^4) \qquad \text{for fixed } j\,, \]
so $\t_{rel} \asymp N^2$, and we take $a_N = N$. Assumption~\ref{a:spectral} holds with
	\[ \g_k = \lim_{N \to \infty} a_N \sqrt{2(1 - \l_k(N))} = \sqrt{2}\,\pi\,\lfloor k/2 \rfloor\,, \qquad k \geq 2\,, \]
and $a_N = N \gg \log N$ for $N$ large.

\begin{Theorem}\label{thm:cycle}
Let $\a_2, \ldots, \a_K : \bbR \to \bbR$ be differentiable functions with bounded derivatives, where $K \geq 2$ is an integer independent of $N$, and fix $y_0 > 0$. Let $\phi : \bbZ_N \times \bbR \to \bbR$ be the test function
	\[ \phi(x, y) = \sum_{k=2}^K \a_k(y)\, f_k^N(x)\,, \]
and let $\varphi(x,y) \coloneqq \phi(x, y/a_N)$ on $\bbZ_N \times \bbZ$. Then with $T = N \lfloor y_0 a_N \rfloor$,
	\[ \frac{1}{N} \laa D_T,\, \varphi \raa
	\Rightarrow \cN\!\left( 0\,, \sum_{k=2}^K \a_k(y_0)^2\, \frac{1 - e^{-2\g_k y_0}}{2 \g_k} \right) \]
as $N \to \infty$, where $\g_k = \sqrt{2}\,\pi\,\lfloor k/2 \rfloor$.
\end{Theorem}

\begin{proof}
This is Theorem~\ref{th:GFF} applied to $V_N = \bbZ_N$ with $a_N = N$: the cycle is vertex-transitive, Assumption~\ref{a:spectral} is verified above, and $a_N = N \gg \log N$.
\end{proof}

This is essentially the GFF result of Jerison, Levine and Sheffield~\cite{jerison2014internal2}, Theorem~3, which is recovered up to the normalization difference between the lazy and non-lazy walks.

\subsubsection{The $d$-dimensional torus $\bbZ_n^d$}
Fix an integer $d \geq 1$ and let $V_N = \bbZ_n^d$ with $N = n^d$. The eigenfunctions of $P_N$ are tensor products of the cycle eigenfunctions~\eqref{eq:fj}, and their eigenvalues have the standard Fourier expansion
	\[ \l_{\underline m}(N) = 1 - \frac{\pi^2}{d\, n^2}|\underline m|^2 + O(n^{-4}) \qquad \text{as } n \to \infty\,, \]
uniformly on every fixed finite set of frequency vectors $\underline m \in \bbZ^d$. Taking $a_N = n = N^{1/d}$, for each fixed $K \geq 2$ the $K$ largest eigenvalues $\l_k(N)$ come from the $K$ frequency vectors $\underline m$ of smallest Euclidean norm (counted with multiplicity): for any fixed such $\underline m$, the $O(n^{-4})$ term is negligible and the ordering of the $\l_{\underline m}(N)$ for $n$ large matches the ordering of $|\underline m|^2$, which is independent of $n$ on any fixed finite set. Assumption~\ref{a:spectral} therefore holds with $\g_k$ the positive square roots of the eigenvalues of the continuum operator $-(1/(2d))\Delta_{\bbT^d}$ (counted with multiplicity), and $a_N = N^{1/d} \gg \log N$ for $N$ large.

\begin{Theorem}\label{thm:torus}
Let $\a_2, \ldots, \a_K : \bbR \to \bbR$ be differentiable functions with bounded derivatives, where $K \geq 2$ is an integer independent of $N$, and fix $y_0 > 0$. Let
	\[ \phi(\underline x, y) = \sum_{k=2}^K \a_k(y)\, f_k^N(\underline x) \]
be the corresponding test function on $\bbZ_n^d \times \bbR$, and let $\varphi(\underline x, y) \coloneqq \phi(\underline x, y/a_N)$ on $\bbZ_n^d \times \bbZ$. Then with $T = N \lfloor y_0 a_N \rfloor$,
	\[ \frac{1}{\sqrt{N a_N}} \laa D_T,\, \varphi \raa
	\Rightarrow \cN\!\left(0\,, \sum_{k=2}^K \a_k(y_0)^2\, \frac{1 - e^{-2\g_k y_0}}{2\g_k} \right) \]
as $n \to \infty$, where $\g_k \coloneqq \lim_{n\to\infty} a_N \sqrt{2(1 - \l_k(N))}$.
\end{Theorem}

\begin{proof}
The base $V_N = \bbZ_n^d$ is vertex-transitive, Assumption~\ref{a:spectral} holds as shown above, and $a_N = N^{1/d} \gg \log N$ for $N$ large. The conclusion follows from Theorem~\ref{th:GFF}.
\end{proof}

For $d=1$, this recovers Theorem~\ref{thm:cycle}. For $d \geq 2$, the result is new: the scalar limits in Theorem~\ref{thm:torus} match the trace at height $y_0$ of the half-cylinder GFF on $\bbT^d \times \bbR_+$ with horizontal generator $-(1/(2d))\Delta_{\bbT^d}$ (see Section~\ref{sec:FGF}).

\subsection{Related results}

\subsubsection{Non-lattice bases}
IDLA has been studied on a range of graphs beyond the lattice $\bbZ^d$. Blach{\`e}re \cite{blachere2004internal} treated discrete groups of polynomial growth, and Blach{\`e}re and Brofferio \cite{blachere2007internal} extended this line to groups of exponential growth. Huss \cite{huss2008internal} proved a shape theorem on non-amenable graphs. Shellef \cite{shellef2010idla} obtained the Euclidean-ball inner bound on supercritical Bernoulli percolation clusters in $\bbZ^d$, and Duminil-Copin, Lucas, Yadin and Yehudayoff \cite{duminil2013containing} completed the picture by proving the corresponding outer bound and full limit-shape theorem.

On comb lattices, Huss and Sava-Huss \cite{huss2012internal} derived an IDLA inner bound via the divisible sandpile as part of a comparative study of three aggregation models, while Asselah and Rahmani \cite{asselah2016fluctuations} subsequently analysed fluctuations around the asymptotic shape. On the pre-fractal Sierpinski gasket graph, Chen, Huss, Sava-Huss and Teplyaev \cite{chen2017internal} proved a shape theorem, and Heizmann \cite{heizmann2023fluctuations} analysed the corresponding fluctuations. Bou-Rabee and Gwynne \cite{bourabee2022harmonic, bourabee2022matedcrt} introduced and analysed IDLA on mated-CRT maps, a family of random planar maps that converge to Liouville quantum gravity surfaces; the limit shape is described by LQG harmonic balls, constructed via an obstacle-problem formulation of Hele--Shaw flow.

\subsubsection{IDLA on cylinders}
Jerison, Levine and Sheffield~\cite{jerison2014internal2} stated a logarithmic maximal-fluctuation bound on the cylinder $\bbZ_N \times \bbZ$: at the quadratic time $T \asymp N^2$, the cluster is trapped between heights $T/N \pm C \log N$ with high probability. Levine and Silvestri~\cite{levine2018long} analysed IDLA on the cycle-base cylinder as a Markov chain on the space of clusters. They proved that, by introducing a suitable shift procedure, IDLA becomes a positive recurrent Markov chain whose stationary distribution concentrates on clusters of logarithmic height in $N$, and they identified the total-variation mixing time  up to logarithmic factors. Silvestri~\cite{silvestri2020internal} extended part of these results to cylinders with quasi-regular base graphs.

\subsubsection{Variants of the dynamics and of the initial configuration}
Beyond changing the base graph, one may vary the source distribution or the walk rule. Levine and Peres~\cite{levine2010scaling} proved continuum scaling limits for IDLA with multiple source points in $\bbZ^d$; Chenavier, Coupier, Penner and Rousselle~\cite{chenavier2024hyperplane} recently studied the hyperplane-source analogue, which is particularly close in spirit to the flat-interface setup of the present paper, and Darrow~\cite{darrow2024extended} treated extended-source IDLA in the plane. In a different direction, Benjamini, Duminil-Copin, Kozma and Lucas~\cite{benjamini2017internal} studied IDLA with particles released uniformly at random from the existing aggregate rather than from a fixed seed, and Asselah, Silvestri and Taggi~\cite{asselah2025internal} recently introduced a critical branching-random-walk variant, identifying a phase transition on the shape theorem according to the dimension.

\subsubsection{Laplacian growth: divisible sandpile, rotor aggregation, and related models}
IDLA is one of several discrete models of Laplacian growth whose scaling limits are characterized by free-boundary and obstacle problems, see~\cite{holroyd2008chip, levine2017laplacian} for surveys. Gravner and Quastel~\cite{gravner2000internal} showed that generalized IDLA with finitely many Poisson sources has a hydrodynamic limit governed by the one-phase Stefan problem, providing one of the earliest direct bridges between IDLA and free-boundary PDEs. Levine and Peres introduced the divisible sandpile~\cite{levine2009strong} as a deterministic analogue of IDLA, and proved~\cite{levine2010scaling} that IDLA, the divisible sandpile, and rotor-router aggregation all share the same continuum scaling limit on $\bbZ^d$ (including the multiple-source case), described by a PDE free-boundary problem. Levine, Murugan, Peres and Ugurcan~\cite{levine2016divisible} established a sharp stabilization/explosion dichotomy for the divisible sandpile with i.i.d.\ initial masses on vertex-transitive graphs, recently extended to general bounded-degree graphs by Bou-Rabee, Peres and Sava-Huss~\cite{bourabee2026divisible}. Pegden and Smart~\cite{pegden2013convergence} proved convergence of the Abelian sandpile on $\bbZ^d$ and characterized the limiting configuration via an elliptic obstacle problem. Bou-Rabee~\cite{bourabee2019convergence, bourabee2021exploding} proved convergence of the Abelian sandpile with random initial conditions. External DLA, the counterpart of IDLA in which particles come from infinity, was introduced by Witten and Sander~\cite{witten1981diffusion}; rigorous results on its continuum limit are known only in specific models, most notably Hastings--Levitov aggregation~\cite{hastings1998laplacian, norris2023scaling, norris2024stability}. A $o(1)$-density result for external DLA on a cylinder with rapidly mixing base graph was shown in \cite{BY08DLA}.

\subsection{Overview of the proof}
The proof of Theorem~\ref{th:GFF} follows the martingale strategy of Jerison, Levine and Sheffield~\cite{jerison2014internal, jerison2014internal2}. The starting point is the observation that, when $\varphi$ is discrete harmonic and has zero horizontal average, the sum $\sum_{(x, y) \in A(t)} \varphi(x, y)$ is a martingale in the particle index $t$: each new particle settles at a random-walk exit site of the current cluster, and by harmonicity its contribution has zero conditional mean. A martingale central limit theorem~\cite{mcleish1974dependent} then yields a Gaussian limit. Implementing this strategy on cylinders over a general vertex-transitive base requires two new ingredients.

The first is a replacement step. Since the original test function $\varphi$ need not be discrete harmonic on the random cluster, we replace it by the discrete harmonic extension of its values
at the cluster's typical height and control the replacement error by an Efron--Stein variance inequality. Exploiting the abelian property of IDLA, resampling a single trajectory changes the cluster by at most two sites (Proposition~\ref{pr:exchange}); combined with vertex-transitivity, the resulting variance bound reduces to a layerwise $L^2$ estimate that closes by orthonormality of the base eigenfunctions alone. Where prior cylinder treatments rely on coupling and mixing-time arguments~\cite{jerison2014internal2, levine2018long, silvestri2020internal}, this step uses only the abelian property, vertex-transitivity, and orthonormality.

The second is a high-probability concentration bound on the IDLA cluster around its typical height (Theorem~\ref{th:max_fluct}). By vertex-transitivity, the horizontal coordinate of the cylinder walk is, conditionally on its vertical trajectory, uniformly distributed on $V_N$ at the first hitting time of each level; combined with the martingale inner-bound technique of Jerison, Levine and Sheffield~\cite{jerison2012logarithmic}, the alive/ghost decomposition of Lawler, Bramson and Griffeath~\cite{lawler1992internal}, and Freedman's martingale inequality~\cite{freedman1975tail}, this yields the bound. The result goes beyond what could be achieved with Silvestri's~\cite{silvestri2020internal} mixing approach and, unlike Levine--Silvestri~\cite{levine2018long}, applies to a wide range of base graphs. 

Together these ingredients reveal that the GFF is a universal scaling limit for IDLA fluctuations on cylinder graphs.

\subsection{Organization of the article}
Section~\ref{sec:apriori} establishes a classical a priori large-deviation bound on the height of IDLA clusters (Proposition~\ref{pr:apriori}). Section~\ref{sec:preliminaries} collects preliminaries on discrete harmonic functions and the martingale structure of IDLA, and proves the replacement Lemma~\ref{le:replace} that allows us to test against a discrete harmonic extension of the original test function. Section~\ref{sec:proof_GFF} proves Theorem~\ref{th:GFF} via the martingale CLT, assuming the maximal fluctuations bound of Theorem~\ref{th:max_fluct}. The remainder of the paper proves Theorem~\ref{th:max_fluct}: Section~\ref{sec:inner} establishes the inner bound by an argument adapted from Jerison, Levine and Sheffield~\cite{jerison2012logarithmic}, using Freedman's martingale inequality; Section~\ref{sec:outer} deduces the outer bound from the inner bound by a coupling argument and the a priori bound from Section~\ref{sec:apriori}; and Section~\ref{sec:reduction} handles the regime $T > T^\sharp$ by an abelian layered-release coupling due to \cite{silvestri2020internal}.

\medskip 

 \section{A priori outer-height bound}\label{sec:apriori}
To start with, we discuss a classical large-deviation bound on the maximal height of IDLA clusters starting from the flat configuration $A(0) = R_0$, which will be used throughout the article. 

\begin{Proposition}\label{pr:apriori}
Let $(A(t))_{t \geq 0}$ be an IDLA process on $V_N \times \bbZ$ starting from $A(0) = R_0$.
\begin{itemize}
\item[(a)] Fix $C \in (0, \infty)$ and $\b > \max(1, e^2 C)$. Then, for every integer $T \leq C N \log N$ and every $N$ large enough (depending on $C$ and $\b$),
	\[ \bbP \big( A(T) \nsubseteq R_{\b \log N} \big) \leq N^{-(\b - 1)}\,. \]
\item[(b)] Fix $\b > e^2$. Then there exists $K_0 = K_0(\b) < \infty$ such that, for every integer $T \geq K_0\, N\log N$ and every $N$ large enough,
	\[ \bbP \big( A(T) \nsubseteq R_{\b T/N} \big) \leq \exp\!\Big( -\b\, \frac{T}{N}\Big)\,. \]
\end{itemize}
\end{Proposition}

\begin{proof}
The argument follows Lemma~C.1 of~\cite{levine2018long}, adapted to the present setting. The structural inputs are: by vertex-transitivity, $P_N$ preserves the uniform measure on $V_N$; walkers are released from the uniform distribution on $V_N \times \{0\}$; and the vertical part of the cylinder kernel~\eqref{def:SRW} is nearest-neighbour, so every path that reaches level $m$ must first pass through level $m-1$.

For each integer $m \geq 1$ and integer $t \geq 0$, let
	\[ Z_m(t) \coloneqq |A(t) \cap (V_N \times \{m\})|\,, \qquad \m_m(t) \coloneqq \bbE[Z_m(t)]\,, \]
denote the number of occupied sites at vertical level $m$ in $A(t)$ and its expectation. 

Set $m_h \coloneqq \lfloor h \rfloor + 1$. Since any particle reaching above height $h$ must pass through all layers at most height $h$, we have that
the event $\{ A(t)  \nsubseteq R_h \}$ implies the event $\{Z_{m_h}(t) \geq 1 \}$.
By Markov's inequality,
	\begin{equation}\label{eq:height-bound}
	 \bbP\!\big( A(t) \nsubseteq R_h \big) = \bbP\!\big( Z_{m_h}(t) \geq 1 \big) \leq \m_{m_h}(t)\,.
	\end{equation}

We now claim that, for every integer $m \geq 1$ and every $t \in \bbN_0$,
	\begin{equation}\label{LBG_up}
	 \m_m(t) \leq \Big( \frac{1}{N} \Big)^{m-1} \binom{t}{m} \leq \Big( \frac{1}{N} \Big)^{m-1}\frac{t^m}{m!}\,.
	\end{equation}

We first establish the recurrence
	\begin{equation}\label{eq:mu-recurrence}
	 \m_m(t+1) - \m_m(t) \leq \frac{1}{N}\, \m_{m-1}(t) \qquad \text{for every integer } m \geq 2 \text{ and } t \geq 0\,.
	\end{equation}
Conditional on $A(t)$, the cluster at time $t+1$ is $A(t) \cup \{ (X_{t+1}, Y_{t+1}) \}$, where $(X_{t+1}, Y_{t+1})$ is the first vertex of the $(t{+}1)$-st walk outside $A(t)$. Hence
	\[ \bbE\!\big[ Z_m(t+1) - Z_m(t) \mid A(t) \big] = \bbP\!\big( Y_{t+1} = m \mid A(t) \big)\,. \]
Let $(\widehat X_s, \widehat Y_s)_{s \geq 0}$ denote the underlying cylinder walk associated with the $(t{+}1)$-st particle (so $(\widehat X_0, \widehat Y_0)$ is uniformly distributed on $V_N \times \{0\}$, and $(X_{t+1}, Y_{t+1})$ is the value of $(\widehat X_s, \widehat Y_s)$ at the first time it lies outside $A(t)$). Since the vertical part of~\eqref{def:SRW} is nearest-neighbour, for the particle to settle at level $m$ the walker must first pass through level $m-1$ at a site that is already in $A(t)$ (otherwise it would settle at level $m-1$). Let
	\[ \tau_{m-1} \coloneqq \inf\{ s \geq 0 : \widehat Y_s = m-1 \} \]
be the first hitting time of level $m-1$ by the underlying walk. By~\eqref{def:SRW}, the vertical coordinate satisfies
	\[ \bbP\big(\widehat Y_{s+1} - \widehat Y_s = 1\big) = \bbP\big(\widehat Y_{s+1} - \widehat Y_s = -1\big) = \tfrac{1}{4}\,, \qquad
	\bbP\big(\widehat Y_{s+1} - \widehat Y_s = 0\big) = \tfrac{1}{2}\,, \]
so $(\widehat Y_s)_{s \geq 0}$ is the lazy symmetric nearest-neighbour random walk on $\bbZ$, which is recurrent. Hence $\tau_{m-1} < \infty$ almost surely, and
	\begin{equation}\label{eq:reach-bound}
	 \bbP\!\big( Y_{t+1} = m \mid A(t) \big) \leq \bbP\!\big( (\widehat X_{\tau_{m-1}}, m-1) \in A(t) \mid A(t) \big)\,.
	\end{equation}

The key observation is that $\widehat X_{\tau_{m-1}}$ is uniformly distributed on $V_N$ regardless of the vertical history. By the cylinder kernel~\eqref{def:SRW}, every step of $(\widehat X_s, \widehat Y_s)$ is either a vertical step (which leaves $\widehat X$ unchanged) or a horizontal step (which updates $\widehat X$ via the kernel $P_N$). Conditionally on the entire vertical trajectory $(\widehat Y_s)_{s \geq 0}$, the horizontal coordinate $(\widehat X_s)$ is a discrete-time Markov chain on $V_N$ that performs the kernel $P_N$ at each horizontal step and stays put at each vertical step. Since $P_N$ preserves the uniform distribution and $\widehat X_0$ is uniform on $V_N$, the conditional distribution of $\widehat X_{\tau_{m-1}}$ given the vertical history is uniform on $V_N$. Marginalizing over the vertical history yields
	\[ \widehat X_{\tau_{m-1}} \text{ is uniformly distributed on } V_N\,, \]
independently of $A(t)$ (which depends only on the first $t$ walks, distinct from the $(t{+}1)$-st). Therefore
	\[ \bbP\!\big( (\widehat X_{\tau_{m-1}}, m-1) \in A(t) \mid A(t) \big)
	= \frac{|A(t) \cap (V_N \times \{m-1\})|}{N}
	= \frac{Z_{m-1}(t)}{N}\,. \]
Combining with~\eqref{eq:reach-bound} and taking expectations gives~\eqref{eq:mu-recurrence}.

We now prove~\eqref{LBG_up} by induction on $m$. For $m = 1$, since $|V_N| = N$ and at each step at most one site is added, $\m_1(t) = \bbE[Z_1(t)] \leq \min(t, N) \leq t = \binom{t}{1}$, which is~\eqref{LBG_up} with $m = 1$.

Assume~\eqref{LBG_up} holds for $m - 1 \geq 1$. Since $A(0) = R_0$ contains no site at vertical level $m \geq 1$, we have $\m_m(0) = 0$. Hence, by~\eqref{eq:mu-recurrence} and a telescoping sum,
	\[ \m_m(t) = \m_m(0) + \sum_{s=0}^{t-1} \big( \m_m(s+1) - \m_m(s) \big)
	\leq \frac{1}{N} \sum_{s=0}^{t-1} \m_{m-1}(s)
	\leq \frac{1}{N^{m-1}} \sum_{s=0}^{t-1} \binom{s}{m-1}\,. \]
Using the hockey stick identity $\sum_{s=0}^{t-1} \binom{s}{m-1} = \binom{t}{m}$, we obtain
	\[ \m_m(t) \leq \frac{1}{N^{m-1}} \binom{t}{m} \leq \frac{1}{N^{m-1}} \frac{t^m}{m!}\,, \]
which completes the induction, proving \eqref{LBG_up}.

For part~(a), apply~\eqref{eq:height-bound} with $h = \b \log N$, so that $m \coloneqq \lfloor \b \log N \rfloor + 1$:
	\[ \bbP\!\big( A(T) \nsubseteq R_{\b\log N} \big) \leq \m_m(T) \leq \frac{1}{N^{m-1}} \frac{T^m}{m!}\,. \]
Using $T \leq CN\log N$ and $m! \geq m^m e^{-m}$,
	\[ \frac{T^m}{N^{m-1} m!} \leq N \cdot \Big( \frac{eT}{Nm} \Big)^m
	\leq N \cdot \Big( \frac{eC \log N}{m} \Big)^m
	\leq N \cdot \Big( \frac{eC}{\b} \Big)^m\,, \]
where we used $m \geq \b \log N$ in the last step. Since $\b \geq e^2 C$, we have $eC/\b \leq e^{-1}$, so the right-hand side is bounded by $N \cdot e^{-m} \leq N \cdot e^{-\b\log N} = N^{1 - \b}$, which is the desired bound.

For part~(b), apply~\eqref{eq:height-bound} with $h = \b T/N$ and $m \coloneqq \lfloor \b T/N \rfloor + 1$:
	\[ \bbP\!\big( A(T) \nsubseteq R_{\b T/N} \big) \leq \m_m(T) \leq \frac{1}{N^{m-1}} \frac{T^m}{m!}\,. \]
Using the identity $N^{-(m-1)} = N \cdot N^{-m}$ together with $m! \geq m^m e^{-m}$ and $m \geq \b T/N$,
	\[ \frac{T^m}{N^{m-1}\, m!}
	= N \cdot \frac{T^m}{N^m\, m!}
	\leq N\, \Big(\frac{e T}{Nm}\Big)^m
	\leq N\, \Big(\frac{e}{\b}\Big)^m
	= \exp\!\Big( \log N - m\,(\log\b - 1) \Big)\,. \]
Since $\b > e^2$ we have $\log\b - 1 > 1$, so using $m \geq \b T/N$ the exponent is bounded above by
	\[
	\log N - \b\,\frac{T}{N}\,(\log\b - 1)
	\;=\; -\b\,\frac{T}{N} \;+\; \Big[\log N - \b\,\frac{T}{N}\,(\log\b - 2)\Big]\,.
	\]
Since $\b > e^2$, we have $\log\b - 2 > 0$. Now take
\[
  K_0 \coloneqq K_0(\b) \coloneqq \frac{1}{\b\,(\log\b - 2)}\,.
\]
For every integer $T \geq K_0 N\log N$, we have $\b(\log\b - 2)\,T/N \geq \log N$, so the bracket is non-positive. Hence the exponent is at most $-\b T/N$, as claimed.
\end{proof}

\begin{Remark}[Uniform binomial form]\label{rk:apriori-uniform}
The argument above in fact establishes the following bound, from which parts~(a) and~(b) drop out by two particular choices of $h$: for every real $h \geq 0$ and every integer $T \geq 0$,
	\[ \bbP\!\big( A(T) \nsubseteq R_h \big) \;\leq\; \frac{1}{N^{\lfloor h \rfloor}}\, \binom{T}{\lfloor h\rfloor + 1}\,. \]
This combines~\eqref{eq:height-bound} and~\eqref{LBG_up}. Part~(a) then follows by taking $h = \b \log N$; part~(b) follows by taking $h = \b T/N$ and additionally using $T \geq K_0 N \log N$ to absorb the leading factor of $N$. In both cases the binomial coefficient is estimated via $\binom{T}{m} \leq T^m/m!$ together with $m! \geq m^m e^{-m}$.
\end{Remark}

\begin{Remark}\label{rk:apriori-Ndep}
The proof of part~(b) extends uniformly to an $N$-dependent exponent. Let $\b_0 > e^2$ be fixed, let $(\b_N)_{N \geq 1}$ be a sequence with $\b_N \geq \b_0$ for every $N$, and let $(T_N)_{N \geq 1}$ be a sequence of integers with $T_N/(N \log N) \to \infty$ as $N \to \infty$. Repeating the proof of part~(b) with the same choice $m_N \coloneqq \lfloor \b_N\, T_N/N\rfloor + 1$ gives
	\begin{align*}
	\bbP\!\big( A(T_N) \nsubseteq R_{\b_N T_N/N} \big)
	&\leq \exp\!\Big( \log N - m_N\,(\log\b_N - 1) \Big) \\
	&\leq \exp\!\Big( {-}\b_N\,\frac{T_N}{N}
	\;+\; \Big[\log N - \b_N\,(\log\b_N - 2)\,\frac{T_N}{N}\Big] \Big)\,.
	\end{align*}
Since the function $g(x) \coloneqq x\,(\log x - 2)$ is increasing on $(e, \infty)$, $\b_N\,(\log\b_N - 2) \geq \b_0\,(\log\b_0 - 2) > 0$ uniformly in $N$. Combined with $T_N/(N \log N) \to \infty$, this forces the bracket
	\[ \log N - \b_N\,(\log\b_N - 2)\,\frac{T_N}{N} \;\longrightarrow\; -\infty\,, \]
so the bracket is non-positive for every $N$ large enough, and therefore
	\[ \bbP\!\big( A(T_N) \nsubseteq R_{\b_N T_N/N} \big)
	\leq \exp\!\Big( {-}\b_N\,\frac{T_N}{N} \Big)\,. \]
We use this $N$-dependent form in the proof of Lemma~\ref{clm:Q on Ec}.
\end{Remark}

\section{Discrete harmonic functions and IDLA}\label{sec:preliminaries}

For simplicity we begin the discussion on the (non-rescaled) cylinder $\bbG_N$, and only later in this section do we rescale the vertical coordinate by a factor $a_N$ to recover the setting of Theorem~\ref{th:GFF}.

\subsection{Discrete harmonic functions and martingales}

Let $(S^X_k, S^Y_k)_{k\geq 0}$ denote the cylinder walk on $\bbG_N$, as defined in~\eqref{def:SRW}. A function $\psi : \bbG_N \to \bbR$ is discrete harmonic if
	\[ \bbE \big[ \psi (S^X_1, S^Y_1) \,\big|\, (S^X_0, S^Y_0) = (x, y) \big] = \psi(x, y) \]
for every $(x, y) \in \bbG_N$.

Summing a discrete harmonic function over an IDLA cluster gives a martingale, provided the function decays sufficiently fast as $y \to -\infty$ for optional stopping to apply to the cylinder walk.

\begin{Lemma}[Harmonic martingale]\label{le:martingale}
Let $\psi : \bbG_N \to \bbR$ be a discrete harmonic function with the zero horizontal average property~\eqref{eq:zero_average} and exponential decay as $y \to -\infty$, let $(A(t))_{t \geq 0}$ be an IDLA process on $\bbG_N$ with $A(0) = R_0$, and let
	\[ M(t) \coloneqq \frac{1}{N}\,\laa D_t, \psi\raa = \frac{1}{N}\sum_{(x,y) \in A_+(t)} \psi(x, y) \qquad (t \geq 0)\,. \]
Then $(M(t))_{t \geq 0}$ is a martingale with respect to the natural filtration of the IDLA process.
\end{Lemma}

\begin{proof}
Conditionally on $A(t)$, the cluster at time $t + 1$ differs from $A(t)$ by the single settlement site $(X_{t+1}, Y_{t+1}) \in \{y \geq 1\}$, so
	\[ M(t+1) - M(t) \;=\; \frac{1}{N}\,\psi(X_{t+1}, Y_{t+1})\,. \]
The particle follows a cylinder walk $(\widehat X_s, \widehat Y_s)_{s \geq 0}$ from $(\widehat X_0, \widehat Y_0)$ distributed as $\pi_N \otimes \d_0$, and settles at the first exit time $\t$ from $A(t)$. Since $A_+(t)$ is finite, $A(t)$ is contained in some half-cylinder $R_h$ with $h$ a finite random integer; and since $R_0 \subseteq A(t)$, the walker does not exit $A(t)$ below level $1$, so it exits at a site in $R_{h+1} \setminus R_0$ and $\tau < \infty$ almost surely. The pre-exit path stays in $A(t) \cup \{(X_{t+1}, Y_{t+1})\} \subseteq R_{h+1}$, on which $\psi$ is bounded: by exponential decay as $y \to -\infty$ and by its finite values at each site of the finite set $V_N \times \{1, \ldots, h+1\}$. The stopped process $(\psi(\widehat X_{s \wedge \tau}, \widehat Y_{s \wedge \tau}))_{s \geq 0}$ is therefore a bounded martingale by discrete harmonicity of $\psi$, and optional stopping gives
\[
  \bbE\big[ \psi(X_{t+1}, Y_{t+1}) \mid A(t) \big] \;=\; \bbE\big[\psi(\widehat X_0, 0) \mid A(t) \big] \;=\; \frac{1}{N}\sum_{x \in V_N} \psi(x, 0) \;=\; 0\,,
\]
where the last equality uses the zero horizontal average property.
\end{proof}

The specific $\psi$ used in the proof of Theorem~\ref{th:GFF} is given in~\eqref{eq:truepsi} below.

\subsection{Discrete harmonic extension away from a given height}
For the proof of Theorem~\ref{th:GFF} we will want to replace a generic test function $\varphi$ on $\bbG_N$ with a discrete harmonic one $\psi$, in such a way that $\psi \equiv \varphi$ on the line $\{y = T/N\}$, around which the cluster's fluctuations are concentrated. We now discuss how to build such a harmonic extension $\psi$.

Recall from~\eqref{eq:eigenvalues} that $(f_k^N)_{k=1}^N$ denotes an orthonormal system of eigenfunctions of $P_N$ with respect to the normalized inner product~\eqref{eq:product}, with corresponding eigenvalues $1 = \l_1(N) > \l_2(N) \geq \cdots \geq \l_N(N) \geq 0$. Throughout this subsection $N$ is fixed, and we abbreviate $\l_k \coloneqq \l_k(N)$. A generic test function $\varphi : \bbG_N \to \bbR$ with the zero horizontal average property~\eqref{eq:zero_average} can be written in this basis as
	\[
	\varphi(x, y) = \sum_{k=2}^N \a_k(y)\, f_k^N(x)
	\]
for some choice of functions $\alpha_k : \bbZ \to \bbR$ with $2 \leq k \leq N$ (the sum starts at $k = 2$ because orthogonality to the constant function $f_1^N \equiv 1$ is exactly the zero horizontal average condition).
\begin{Proposition} \label{pr:hextension}
For $2 \leq k \leq N$, let $q_k > 0$ be the positive real number defined by
	\begin{equation}\label{def:qk}
	 \cosh(q_k) = 2 - \l_k\,.
	 \end{equation}
(By irreducibility and laziness, $\l_k \in [0, 1)$ for $k \geq 2$, so $2 - \l_k \in (1, 2]$ and $q_k > 0$.) Then, for every $h \in \bbZ$ and every sequence $(\a_k(h))_{2 \leq k \leq N}$ of real numbers, the function
	\[ \psi(x, y) = \sum_{k=2}^N \a_k(h)\, f_k^N(x)\, e^{q_k(y - h)} \]
is discrete harmonic on $\bbG_N$, decays exponentially as $y \to -\infty$, and satisfies $\psi(x, h) = \sum_{k=2}^N \a_k(h)\,f_k^N(x)$ on the line $\{(x, h) : x \in V_N\}$.
\end{Proposition}

\begin{proof}
The boundary identity $\psi(x, h) = \sum_k \a_k(h)\,f_k^N(x)$ is immediate, and exponential decay as $y \to -\infty$ follows from $q_k > 0$. To check discrete harmonicity, fix $(x, y) \in \bbG_N$ and apply the cylinder kernel~\eqref{def:SRW} to $\psi$:
\[
	\sum_{(x', y')} \cP_N\big((x, y); (x', y')\big)\, \psi(x', y')
	= \tfrac{1}{4}\,\psi(x, y+1) + \tfrac{1}{4}\,\psi(x, y-1) + \tfrac{1}{2} \sum_{x'} P_N(x, x')\, \psi(x', y)\,.
\]
Using $\tfrac{1}{2}(e^{q_k} + e^{-q_k}) = \cosh(q_k)$ and $\sum_{x'} P_N(x, x') f_k^N(x') = \l_k\, f_k^N(x)$, this equals
\[
	\sum_k \a_k(h)\, f_k^N(x)\, e^{q_k(y - h)} \cdot \tfrac{1}{2}\big( \cosh(q_k) + \l_k \big)\,.
\]
By~\eqref{def:qk}, $\cosh(q_k) + \l_k = 2$, so the sum equals $\psi(x, y)$.
\end{proof}

\subsection{Rescaling the vertical coordinate}

As in the statement of Theorem~\ref{th:GFF}, let $\phi : V_N \times \bbR \to \bbR$ be a test function of the form
	\[ \phi(x, y) = \sum_{k=2}^K \alpha_k(y)\, f_k^N(x) \]
for some finite integer $K \geq 2$, where each coefficient $\alpha_k$ is differentiable on a neighbourhood of $y_0$ with bounded derivative there (note that such a function satisfies the zero horizontal average property~\eqref{eq:zero_average}). We extend this to a function $\varphi$ on $\bbG_N$ by setting
	\[ \varphi(x, y) \coloneqq \phi\!\Big( x, \frac{y}{a_N} \Big) \]
for every $(x, y) \in \bbG_N$, with $a_N$ as in~\eqref{eq:aN}, so that $a_N \asymp \sqrt{\t_{rel}} \gg 1$ as $N \to \infty$. Thus
	\[ \varphi(x, y) = \sum_{k=2}^K \a_k\!\Big( \frac{y}{a_N} \Big)\, f_k^N(x) \]
for every $(x, y) \in \bbG_N$.

For each $2 \leq k \leq K$, define $q_k^N > 0$ by
	\begin{equation}\label{def:qkN}
	 \cosh\!\left( \frac{q_k^N}{a_N} \right) = 2 - \l_k\,.
	 \end{equation}
(By irreducibility, $\l_k < 1$ for $k \geq 2$, so $2 - \l_k > 1$ and $q_k^N > 0$.)
By Proposition~\ref{pr:hextension} (applied with the rescaled exponent $q_k^N/a_N$ in place of $q_k$), the function
	\begin{equation}\label{eq:truepsi}
	 \psi(x, y) \coloneqq \sum_{k=2}^K \a_k\!\big( \tfrac{T}{Na_N} \big)\, f_k^N(x)\, 
	 \exp \Big( \tfrac{q_k^N (y- \tfrac{T}{N} ) }{a_N } \Big) 
	 \end{equation}
is discrete harmonic on $\bbG_N$, decays as $y \to -\infty$, and satisfies $\psi \equiv \varphi$ on the line $\{(x, T/N) : x \in V_N\}$. By the modewise spectral hypothesis of Theorem~\ref{th:GFF}, for each $2 \leq k \leq K$ we have $q_k^N / a_N \to 0$ and $q_k^N \to \g_k$ as $N \to \infty$.


To replace $\varphi$ by $\psi$ in~\eqref{eq:thGFF}, we use an Efron--Stein variance inequality with a layerwise $L^2$ estimate, exploiting both the abelian property of IDLA and vertex-transitivity.

Realize IDLA from independent infinite walk trajectories
$\o = (\o_1,\ldots,\o_T)$,
where $\o_j$ starts from a point of $V_N\times\{0\}$ chosen according to $\pi_N \otimes \d_0$
and follows the cylinder walk forever.
The $j$-th particle settles at the first site of $\o_j$ outside the current cluster.
We use $A_+(T;\o)$ to denote the cluster obtained for the random walks $(\o_1,\ldots, \o_T)$.

\begin{Proposition}[Stability under trajectory swap]\label{pr:exchange}
If\/ $\o$ and $\widetilde\o$ differ in exactly one coordinate,
then $|A_+(T;\o)\triangle A_+(T;\widetilde\o)|\leq 2$.
\end{Proposition}

\begin{proof}
Suppose $\o$ and $\widetilde\o$ differ only in coordinate $j$. By the abelian property~\cite{diaconis1991growth}, we may process trajectory $j$ last without changing the final cluster. After processing the $T - 1$ shared trajectories $\o_1, \ldots, \o_{j-1}, \o_{j+1}, \ldots, \o_T$, the two clusters agree; call the common result $C$. Processing the last trajectory then gives
	\[ A(T;\o) = C \cup \{z_j\}\,, \qquad A(T;\widetilde\o) = C \cup \{\tilde z_j\}\,, \]
where $z_j \notin C$ is the first exit of $\o_j$ from $C$ and $\tilde z_j \notin C$ is the first exit of $\widetilde\o_j$ from $C$. Since $C \supseteq R_0$, both $z_j$ and $\tilde z_j$ lie in $\{y \geq 1\}$, so $A_+(T;\o) \triangle A_+(T;\widetilde\o) \subseteq \{z_j, \tilde z_j\}$.
\end{proof}

\begin{Lemma}\label{le:replace}
Fix $y_0 > 0$ and $\n > 0$, and assume $a_N \gg \log N + \log_+ a_N$ as $N \to \infty$ and that Assumption~\ref{a:spectral} holds for each $2 \leq k \leq K$. Let $T \coloneqq N \lfloor a_N y_0\rfloor$, $h \coloneqq T/N$, and let $(\Delta_N)_{N \geq 1}$ be positive reals satisfying $\Delta_N = O\!\left(\sqrt{a_N(\log N + \log_+ a_N)}\right)$ and
	\[ \bbP\big(R_{\lfloor h - \Delta_N\rfloor} \subseteq A(T) \subseteq R_{h + \Delta_N}\big) \geq 1 - N^{-\n} \qquad \text{for } N \text{ large enough}\,. \]
Let $\phi : V_N \times \bbR \to \bbR$ be a test function of the form $\phi(x, y) = \sum_{k=2}^K \a_k(y)\, f_k^N(x)$, where each $\a_k : \bbR \to \bbR$ is independent of $N$ and has bounded derivative in a neighbourhood of $y_0$; let $\varphi(x, y) \coloneqq \phi(x, y/a_N)$, and let $\psi$ be the discrete harmonic extension of $\varphi$ defined in~\eqref{eq:truepsi}. Then
	\[ \frac{1}{\sqrt{Na_N}}\,\laa D_T, \varphi - \psi\raa \;\longrightarrow\; 0 \qquad \text{in probability as } N \to \infty\,. \]
\end{Lemma}

\begin{proof}
By the hypothesis, the good event
	\[ E \coloneqq \{ R_{\lfloor h - \Delta_N\rfloor} \subseteq A(T) \subseteq R_{h + \Delta_N} \} \]
satisfies $\bbP(E^c) \leq N^{-\n}$ for $N$ large enough. Set $F_N \coloneqq R_{h + \Delta_N} \setminus R_{\lfloor h - \Delta_N\rfloor}$, a union of full horizontal layers contained in $\{y \geq 1\}$ for $N$ large (since $h \asymp y_0 a_N$ and $\Delta_N = o(a_N)$). On $E$ we have $A(T) \triangle R_h \subseteq F_N$. Both $\varphi$ and $\psi$ are linear combinations of $\{f_k^N\}_{k \geq 2}$, hence have zero horizontal average on every layer, so $\sum_{x \in V_N}(\varphi - \psi)(x, y) = 0$ for every $y \in \bbZ$. This makes $\laa D_T, \varphi - \psi\raa$ insensitive to full horizontal layers: the contribution from $R_h$ vanishes layer by layer, and the contribution from layers of $A(T)$ at heights $\leq \lfloor h - \Delta_N\rfloor$ also vanishes (these are full on $E$). The remaining contribution comes from $A(T) \cap F_N$, giving
	\[ \laa D_T, \varphi - \psi \raa
	= \sum_{(x,y) \in A_+(T)} (\varphi - \psi)(x,y) \cdot \mathbf{1}_{F_N}(x,y)
	\qquad \text{on } E\,. \]
Write
	\[ \Phi_N(x, y) \coloneqq (\varphi - \psi)(x, y) \cdot \mathbf{1}_{F_N}(x, y)\,, \]
	\[ S_N \coloneqq \frac{1}{\sqrt{Na_N}} \sum_{(x, y) \in A_+(T)} \Phi_N(x, y)\,, \qquad
	Z_N \coloneqq \frac{1}{\sqrt{Na_N}}\,\laa D_T, \varphi - \psi\raa\,. \]
Then $Z_N \mathbf{1}_E = S_N \mathbf{1}_E$. We show below that $\bbE[S_N^2] \to 0$; since $\bbP(E^c) \to 0$, this implies that for every $\varepsilon > 0$,
\begin{equation}\label{eq:Zn-decomp}
	\bbP\big(|Z_N| > \varepsilon\big) \;\leq\; \bbP(E^c) + \bbP\big(|S_N| > \varepsilon\big)
	\;\leq\; \bbP(E^c) + \varepsilon^{-2}\,\bbE[S_N^2] \;\longrightarrow\; 0\,,
\end{equation}
so $Z_N \to 0$ in probability.

Expand $\varphi - \psi$ in the eigenbasis: with
	\begin{equation}\label{eq:betak}
	\b_k(y) \coloneqq \a_k \big( \tfrac{y}{a_N} \big) - \a_k \big( \tfrac{T}{Na_N} \big) \, \exp \Big( \tfrac{q_k^N (y - \tfrac{T}{N} ) }{ a_N } \Big) \,,
	\end{equation}
we have $(\varphi - \psi)(x,y) = \sum_{k=2}^K \b_k(y)\, f_k^N(x)$. Note that $\b_k(h) = 0$, and the mean value theorem gives
	\begin{equation}\label{eq:betabd}
	|\b_k(y)| \leq C\,\frac{\Delta_N}{a_N} \qquad \text{for } |y - h| \leq \Delta_N\,,
	\end{equation}
where $C$ denotes a constant depending on $K$, the $\a_k$'s, and the $\g_k$'s, but not on $N$; $C$ may change from line to line.

Since $F_N$ is a union of full horizontal layers, $\Phi_N(\cdot, y) = 0$ whenever the layer at height $y$ lies outside $F_N$. For layers inside $F_N$, orthonormality of $\{f_k^N\}$ in the normalized inner product~\eqref{eq:product} gives
	\[ \frac{1}{N} \sum_{x \in V_N} |(\varphi - \psi)(x, y)|^2
	= \sum_{k=2}^K |\b_k(y)|^2\,. \]
Together with~\eqref{eq:betabd}, this yields the layerwise $L^2$ bound
	\begin{equation}\label{eq:layerL2}
		\sup_{y \in \bbZ}\, \frac{1}{N} \sum_{x \in V_N} |\Phi_N(x, y)|^2
		\;\leq\; C \cdot \Big(\frac{\Delta_N}{a_N}\Big)^2\,.
	\end{equation}

For each integer $1 \leq j \leq T$, obtain $\widetilde\o$ from $\o$ by replacing the $j$-th trajectory with an independent copy $\widetilde\o_j$, and write
	\[ D_j \coloneqq A_+(T;\o) \triangle A_+(T;\widetilde\o)\,. \]
By Proposition~\ref{pr:exchange}, $|D_j| \leq 2$ pathwise. Cauchy--Schwarz then gives
	\begin{equation}\label{eq:CS}
	\Big( \sum_{(x,y) \in A_+(T;\o)} \!\!\Phi_N(x,y) \;-\!\! \sum_{(x,y) \in A_+(T;\widetilde\o)} \!\!\Phi_N(x,y) \Big)^2
	\;\leq\; 2 \sum_{u \in D_j} |\Phi_N(u)|^2\,.
	\end{equation}

We use vertex-transitivity to bound the expectation of the right-hand side of~\eqref{eq:CS}. For every automorphism $\g$ of $V_N$, the lift $\bar\g(x, y) \coloneqq (\g x, y)$ preserves the cylinder kernel~\eqref{def:SRW} and the start distribution $\pi_N \otimes \d_0$. By induction on the particle index, $A_+(T; \bar\g \o) = \bar\g\, A_+(T; \o)$ pathwise. Since resampling commutes with applying $\bar\g$ coordinatewise, the joint law of $(\o, \widetilde\o)$ is $\bar\g$-invariant, and so is the law of $D_j$. By transitivity, $\bbP\bigl((x, y) \in D_j\bigr)$ depends only on $y$. Writing this common value as $r_j(y)$ and using $\sum_y N\,r_j(y) = \bbE|D_j| \leq 2$, we obtain
	\begin{equation}\label{eq:layeravg}
	 \bbE\!\bigg[ \sum_{u \in D_j} |\Phi_N(u)|^2 \bigg]
	= \sum_y N\,r_j(y) \cdot \frac{1}{N}\sum_{x} |\Phi_N(x,y)|^2
	\leq 2 \sup_y \frac{1}{N}\sum_{x} |\Phi_N(x,y)|^2\,.
	\end{equation}
The Efron--Stein inequality, combined with~\eqref{eq:CS},~\eqref{eq:layeravg}, and~\eqref{eq:layerL2}, gives
	\begin{equation}\label{eq:varSN}
	\mathrm{Var}\!\Big(\sum_{(x,y) \in A_+(T)} \Phi_N\Big)
	\;\leq\; 2\, T \sup_y \frac{1}{N}\sum_{x} |\Phi_N(x,y)|^2
	\;\leq\; C\, T\, \Big(\frac{\Delta_N}{a_N}\Big)^2\,.
	\end{equation}
Dividing by $Na_N$, using $T \asymp Na_N$ and $\Delta_N^2 = O(a_N(\log N + \log_+ a_N))$, this yields $\mathrm{Var}(S_N) = O\!\big((\log N + \log_+ a_N)/a_N\big) = o(1)$.

It remains to show that the mean vanishes. By the same horizontal automorphism invariance applied to $A(T)$ alone, $\bbP\bigl((x, y) \in A(T)\bigr)$ depends only on $y$. Since each $f_k^N$ with $k \geq 2$ satisfies $\sum_{x \in V_N} f_k^N(x) = 0$ (orthogonality to $f_1^N \equiv 1$), and since $F_N$ is a union of full horizontal layers, $\sum_{x \in V_N} \Phi_N(x, y) = 0$ for every $y \in \bbZ$. Therefore
	\[ \bbE\!\bigg[ \sum_{(x,y) \in A_+(T)} \Phi_N(x,y) \bigg]
	\;=\; \sum_{y} \bbP\bigl((\cdot, y) \in A(T)\bigr) \sum_{x \in V_N} \Phi_N(x, y)
	\;=\; 0\,, \]
so $\bbE[S_N] = 0$ exactly. Combining, $\bbE[S_N^2] = \mathrm{Var}(S_N) = o(1)$. Together with~\eqref{eq:Zn-decomp}, this gives $Z_N \to 0$ in probability as $N \to \infty$.
\end{proof}

This shows that, in order to prove Theorem~\ref{th:GFF}, it suffices to prove that the sequence of random variables
	\[ \frac{1}{\sqrt{Na_N}} \laa D_T, \psi\raa
	= \frac{1}{\sqrt{Na_N}} \sum_{(x,y) \in A_+(T)} \psi(x, y) \]
converges in distribution to a centred Gaussian random variable as $N \to \infty$. The equality between the two expressions holds because $\psi$ has zero horizontal average on every level (it is a finite linear combination of the eigenfunctions $f_k^N$ for $k \geq 2$, which are orthogonal to $f_1^N \equiv 1$): each complete horizontal layer in $R_{T/N}\setminus R_0$ contributes zero to the pairing, so
	\[ \laa D_T, \psi\raa = \sum_{(x,y)\in A_+(T)} \psi(x, y) - \sum_{(x,y) \in R_{T/N}\setminus R_0} \psi(x, y) = \sum_{(x,y) \in A_+(T)} \psi(x, y)\,. \]
Moreover, the process $t \mapsto (Na_N)^{-1/2}\sum_{(x,y)\in A_+(t)} \psi(x, y)$ is a martingale with respect to the natural filtration of the IDLA process. Indeed, $\psi$ is discrete harmonic by Proposition~\ref{pr:hextension}, has zero horizontal average, and decays exponentially as $y \to -\infty$. The exponential decay is the admissibility condition of Lemma~\ref{le:martingale}, and it holds because $q_k^N > 0$ for every $2 \leq k \leq K$ by irreducibility.

\medskip 

\section{Proof of Theorem \ref{th:GFF}}\label{sec:proof_GFF}
In this section we prove Theorem~\ref{th:GFF}, assuming Theorem~\ref{th:max_fluct} and Proposition~\ref{pr:apriori}. Theorem~\ref{th:max_fluct} itself is proved later in Section~\ref{sec:max_fluct}. 

Fix $\n > 0$ throughout this section.

Since $V_N$ is connected and vertex-transitive, its relaxation time is polynomial in $N$ (for instance, $\t_{rel} = O(N^2)$ by a standard diameter estimate), so $a_N \asymp \sqrt{\t_{rel}}$ is polynomial in $N$ and $T = N\lfloor y_0 a_N\rfloor \leq N^m$ for some fixed integer $m$ and all large enough $N$. We apply Theorem~\ref{th:max_fluct} at this $m$ throughout.

Let $\phi : V_N \times \bbR \to \bbR$ be a test function of the form
$$ \phi(x, y) = \sum_{k=2}^K \alpha_k(y)\, f_k^N(x)\,, $$
as in the statement of Theorem~\ref{th:GFF}, with each $\alpha_k$ differentiable with bounded derivative on a neighbourhood of $y_0$. Let $\varphi(x, y) \coloneqq \phi(x, y/a_N)$ and let $\psi$ be its discrete harmonic extension from the line $y = T/N$ as in~\eqref{eq:truepsi}. Note that $T/N \asymp a_N$ and $T^\sharp/N = \sqrt{\t_{mix}}\,(\log N)^2 \gg a_N$ (using $\t_{mix} \geq c(\t_{rel} - 1)$ for an absolute constant $c > 0$, see~\cite{levin2017markov}), so $T \leq T^\sharp$ for $N$ large and the width in Theorem~\ref{th:max_fluct} reduces to $\Delta_N = O(\sqrt{a_N(\log N + \log_+ a_N)}) = o(a_N)$ under our standing assumption $a_N \gg \log N + \log_+ a_N$. Therefore Lemma~\ref{le:replace} applies, and $(Na_N)^{-1/2}\laa D_T, \varphi - \psi\raa \to 0$ in probability. By Slutsky's theorem, it therefore suffices to show that
	\[  \frac{1}{\sqrt{N a_N}} \laa D_T, \psi \raa = \frac{1}{\sqrt{N a_N}} \sum_{(x,y) \in A_+(T)} \psi (x,y)\]
converges to a centred Gaussian random variable with variance $\sigma^2(\phi)$ from Theorem~\ref{th:GFF}; the equality uses the zero horizontal average of $\psi$. To this end we use the following version of the Martingale Central Limit Theorem due to McLeish~\cite{mcleish1974dependent}.

\begin{Theorem}[\cite{mcleish1974dependent}, Theorem 2.3]
\label{thm:martingale-CLT}
Let $(\cX_{k,n})_{k\leq t_n } $ be a martingale difference array with respect to $(\cF_{k,n} )_{k\leq t_n }$. Assume that: 
\begin{itemize}
\item[(a)] $\displaystyle \bbE \Big( \max_{k\leq t_n} \cX_{k,n}^2 \Big) $ is uniformly bounded in $n$, 
\item[(b)] $\displaystyle \max_{k\leq t_n} | \cX_{k,n} | \to 0 $ in probability as $n\to\infty$, 
\item[(c)] $\displaystyle \sum_{k=1}^{t_n} \cX_{k,n}^2 \to \sigma^2 $ 
in probability as $n\to\infty$, for some $\sigma^2 \in (0,\infty )$.
\end{itemize}
Then, $\displaystyle \sum_{k=1}^{t_n} \cX_{k,n} \to \cN (0,\sigma^2)$ in distribution as $n\to\infty$. 
\end{Theorem}
We apply Theorem~\ref{thm:martingale-CLT} with row index $N$ in place of $n$. Set $t_N \coloneqq T = N\lfloor y_0 a_N\rfloor$, $\cF_{k,N} \coloneqq \cF_k$ where $\cF_k = \sigma(A(0), \ldots, A(k))$ is the natural filtration of the IDLA process, and
	\[ \cX_{k,N} \coloneqq \frac{1}{\sqrt{Na_N}}\, \psi(X_k, Y_k)\,, \qquad 1 \leq k \leq t_N\,, \]
where $(X_k, Y_k) \in V_N \times \bbZ$ is the settling location of the $k$-th walker, that is, the unique vertex with $A(k) \setminus A(k-1) = \{ (X_k, Y_k) \}$. Then
	\[ M_N(t) \coloneqq \sum_{k=1}^{t} \cX_{k,N}
	= \frac{1}{\sqrt{Na_N}} \sum_{(x,y) \in A_+(t)} \psi(x,y) \]
defines a martingale indexed by integer times $t \in \{0, 1, \ldots, t_N\}$ by the martingale statement of Section~\ref{sec:preliminaries}, applied to the function $\psi$ from~\eqref{eq:truepsi}: $\psi$ is discrete harmonic (Proposition~\ref{pr:hextension}), has zero horizontal average, and decays exponentially as $y \to -\infty$ (since every $q_k^N > 0$ for $2 \leq k \leq K$), which is the admissibility condition needed for optional stopping. We must verify the assumptions~(a),~(b),~(c) of Theorem~\ref{thm:martingale-CLT} for the array $(\cX_{k,N})$.

Define
$$ Q_N \coloneqq \sum_{k=1}^{t_N} \cX_{k,N}^2 =
\frac{1}{N a_N} \sum_{k=1}^T | \psi(X_k,Y_k) |^2 =
\frac{1}{N a_N} \sum_{(x,y) \in A_+(T) } |\psi(x,y)|^2\,, $$
$$ W_N \coloneqq \frac{1}{N a_N} \sum_{(x,y) \in R_{T/N} \setminus R_0 } | \psi(x,y)|^2  $$
and
\[ \sigma^2 (\phi )  = \sum_{k=2}^K \alpha_k(y_0)^2 \bigg( \frac{1-e^{-2\gamma_k y_0}}{2\gamma_k} \bigg). \]

Let $\Delta_N$ be the width from Theorem~\ref{th:max_fluct} applied at $T = N\lfloor y_0 a_N\rfloor$; under the standing assumption $a_N \gg \log N + \log_+ a_N$, the square-root term dominates the logarithmic term, so $\Delta_N = O(\sqrt{a_N(\log N + \log_+ a_N)}) = o(a_N)$. The good event
	\[ E \coloneqq \big\{ R_{\lfloor T/N - \Delta_N\rfloor} \subseteq A(T) \subseteq R_{T/N + \Delta_N} \big\} \]
satisfies $\bbP(E) \geq 1 - N^{-\n}$ for $N$ large enough. Since $T/(N\log N) = \lfloor y_0 a_N\rfloor / \log N \to \infty$, the $N$-dependent form of Proposition~\ref{pr:apriori}-(b) (Remark~\ref{rk:apriori-Ndep}) applies at time $T$, and we use it freely in what follows.

As a first step, we show in Lemma~\ref{clm:Q on Ec} that the bad event $E^c$ contributes negligibly to $Q_N$ in $L^1$.

\begin{Lemma}
\label{clm:Q on Ec}
We have $\bbE(Q_N \mathbf{1}_{E^c}) \to 0$ as $N \to \infty$.
\end{Lemma}
\begin{proof}
Choose 
\[ J =  3 + \bigg\lceil \frac{\log N}{\log ( a_N / \sqrt{\log N } ) } \bigg\rceil , \]
so that $( \frac{a_N}{\sqrt{\log N} } )^{J-3} \geq N$.
Let $h_2 = \lceil a_N \sqrt{\log N } \rceil$ and inductively define 
$$ h_j = h_{j-1} \cdot \left \lceil \frac{ a_N}{ \sqrt{ \log N } } \right \rceil \qquad 
\mbox{ for } 3 \leq j \leq J . $$
Note that, since we assumed that $a_N \gg \log N$, the sequence $(h_j)_{j=2}^J  $ is increasing,
and $h_j > h_2 \geq e^3 T / N$ for all $3 \leq j \leq J$.

For $2 \leq j \leq J$ define $E_j = \{ A(T) \subseteq R_{h_j} \}$ and write $E_1 = E$ for consistency. Note that $E_2 \subseteq E_3 \subseteq \cdots \subseteq E_J$ by the monotonicity $h_j \leq h_{j+1}$, and $E_1 \subseteq E_2$ for $N$ large because $h_2 \geq e^3 T/N \geq T/N + \Delta_N$ for $N$ large. So the sequence $(E_j)_{j=1}^J$ is nested.
By Proposition~\ref{pr:apriori}-(b), applied with the $N$-dependent exponent $\b_j \coloneqq h_j N/T \geq e^3$ (which is $\geq e^3 > e^2$, so Remark~\ref{rk:apriori-Ndep} applies), we have $\bbP(E_j) \geq 1 - e^{-\b_j T/N} = 1 - e^{-h_j}$ for all $2 \leq j \leq J$.
Combining this with Theorem~\ref{th:max_fluct} we get
	\[ \begin{split} 
	& \bbP(E_2 \setminus E_1 )  \leq \bbP (E_1^c ) \leq N^{-\nu} 
	\\  
	& \bbP(E_j \setminus E_{j-1}  )  \leq \bbP (E_{j-1}^c ) \leq e^{-h_{j-1} } , \qquad 
	3\leq j \leq J. 
	\end{split} \] 
	
Now, using the fact that $(f^N_k)_k$ form an orthonormal basis, we see that for any $y \in \bbZ$

	\begin{align} \label{eqn:psi squared}
		\sum_{x \in V_N} |\psi(x,y)|^2 
		& = \sum_{j,k=2}^K \alpha_j \Big( \frac{T}{Na_N} \Big) \alpha_k \Big( \frac{T}{Na_N} \Big) 
		e^{(q_j^N + q_k^N) \big( \frac{y}{a_N} - \frac{T}{N a_N} \big)} 
		\sum_{x \in V_N} f_j^N(x) f_k^N(x) \nonumber \\
		& = N \sum_{k=2}^K \alpha_k^2\!\Big( \frac{T}{Na_N} \Big) 
		\cdot e^{2 q_k^N \big( \frac{y}{a_N} - \frac{T}{N a_N} \big) } .
	\end{align}

Now fix an integer $h \geq 1$. On the event $\{A(T) \subseteq R_h\}$, every $(x, y) \in A_+(T)$ satisfies $y \leq h$, so
\[
  Q_N \;=\; \frac{1}{N a_N} \sum_{(x,y) \in A_+(T)} |\psi(x, y)|^2
  \;\leq\; \frac{1}{N a_N} \sum_{(x,y) \in R_h \setminus R_0} |\psi(x, y)|^2\,.
\]
The right-hand side is a deterministic quantity, which we estimate via~\eqref{eqn:psi squared}. For each $k$, using $e^{-2 q_k^N T/(N a_N)} \leq 1$ to drop the constant factor, the geometric sum is bounded by
	\[ \sum_{y=1}^h e^{2 q_k^N ( y/a_N - T/(N a_N) )}
	\;=\; e^{-2 q_k^N T/(N a_N)} \sum_{y=1}^h e^{2 q_k^N y/a_N}
	\;\leq\; \frac{e^{2 q_k^N (h+1)/a_N}}{e^{2 q_k^N/a_N} - 1}\,, \]
and therefore
\begin{align}
\frac{1}{N a_N} \sum_{(x,y) \in R_h \setminus R_0} | \psi(x,y)|^2
& = \frac{1}{a_N} \sum_{k=2}^K \alpha_k^2\!\Big( \frac{T}{N a_N} \Big)
\sum_{y=1}^h e^{2 q_k^N ( \frac{y}{a_N} - \frac{T}{N a_N} ) } \nonumber \\
& \leq \sum_{k=2}^K \alpha_k^2\!\Big( \frac{T}{N a_N} \Big)
\cdot \frac{1}{ a_N (e^{2 q_k^N / a_N} - 1) } \cdot e^{2 q_k^N (h+1) / a_N } .
\label{eqn:Eh}
\end{align}

Under our standing assumptions, $T/(N a_N) \to y_0$, so $\alpha_k(T/(Na_N)) = O(1)$; and $q_k^N \to \g_k \in (0, \infty)$, so $a_N\,(e^{2 q_k^N / a_N} - 1) \to 2\g_k$ is bounded away from $0$ and $\infty$.

Take $j = 2$. Since $h_2/a_N = O(\sqrt{\log N})$ and $q_k^N$ is bounded, the exponential in~\eqref{eqn:Eh} is at most $e^{C\sqrt{\log N}} = N^{o(1)}$. Combined with $\bbP(E_2 \setminus E_1) \leq N^{-\n}$,
\begin{align}
\bbE ( Q_N \mathbf{1}_{E_2 \setminus E_1} )
& \leq  \frac{1}{N a_N} \sum_{ (x,y) \in R_{h_2} \setminus R_0 } | \psi(x,y)|^2 \cdot \bbP ( E_2 \setminus E_1 )  \nonumber
\\
& \leq C \sum_{k=2}^K \alpha_k^2\!\Big( \frac{T}{N a_N} \Big) \cdot e^{2 q_k^N (h_2+1)/ a_N }
\cdot N^{-\nu }   \leq C' N^{-\nu/2}
\label{eqn:j=2}
\end{align}
for $N$ large enough (absorbing the subpolynomial factor $e^{C \sqrt{\log N}}$ into the $N^{-\n/2}$ slack).

For $j \geq 3$, the recurrence $h_j = h_{j-1}\lceil a_N/\sqrt{\log N}\rceil$ gives $h_j/a_N \leq 2\,h_{j-1}/\sqrt{\log N}$ for $N$ large, so the exponential in~\eqref{eqn:Eh} with $h = h_j$ is at most $\exp(C'\,h_{j-1}/\sqrt{\log N})$. Since $\bbP(E_j \setminus E_{j-1}) \leq e^{-h_{j-1}}$,
\begin{align}
\bbE ( Q_N \mathbf{1}_{E_j \setminus E_{j-1}} )
& \leq  \frac{1}{N a_N} \sum_{ (x,y) \in R_{h_j} \setminus R_0 } | \psi(x,y)|^2 \cdot \bbP ( E_j \setminus E_{j-1} ) \nonumber \\
& \leq C' \exp \Big( C' \frac{ h_{j-1} }{ \sqrt{ \log N} } -  h_{j-1} \Big) \leq C e^{- h_{j-1} / 2 } ,
\label{eqn:j>2}
\end{align}
using $C'/\sqrt{\log N} \leq 1/2$ for $N$ large.

By the recurrence $h_j = h_{j-1}\lceil a_N/\sqrt{\log N}\rceil$ iterated $J-2$ times,
$$ h_J \;\geq\; h_2 \cdot \Big(\frac{a_N}{\sqrt{\log N}}\Big)^{J-2} \;\geq\; h_2 \cdot N\,, $$
by our choice $J = 3 + \lceil \log N / \log(a_N/\sqrt{\log N})\rceil$. Since $h_2 \geq e^3 T/N \geq T/N$, this gives $h_J \geq N \cdot T/N = T$.
By the ancestor-chain argument used in the proof of Proposition~\ref{pr:apriori} (which uses nearest-neighbour vertical steps of the cylinder kernel~\eqref{def:SRW}), the maximal height of $A(T)$ grows by at most one per added particle, hence $A(T) \subseteq R_T \subseteq R_{h_J}$ deterministically, so $\bbP(E_J) = 1$.
Combining \eqref{eqn:j=2} and \eqref{eqn:j>2},
$$ \bbE ( Q_N \mathbf{1}_{E_1^c} ) = \sum_{j=2}^J \bbE ( Q_N \mathbf{1}_{E_j \setminus E_{j-1} }  )
\leq C N^{-\nu/2} + C J e^{- h_2 /2 }\,. $$
Since $h_2 \geq a_N \sqrt{\log N}$ and $a_N \gg \log N$, we have $h_2 / \log N \to \infty$. On the other hand, $J = 3 + \lceil \log N / \log(a_N/\sqrt{\log N}) \rceil = O(\log N)$. Hence $J\, e^{-h_2/2} \to 0$ (super-polynomially in $N$), and combined with $N^{-\nu/2} \to 0$ this gives $\bbE(Q_N \mathbf{1}_{E_1^c}) \to 0$ as $N \to \infty$.
\end{proof}

We next show that $Q_N$ and $W_N$ converge to the same limit.
\begin{Lemma}
\label{clm:Q-W}
There exists a constant $C>0$ (which may depend on the functions $(\alpha_k)_k$
and $y_0, \nu$, but is independent of $N$), such that
$|Q_N - W_N|\, \mathbf{1}_{E} \leq C \frac{\Delta_N}{a_N} \to 0$
as $N \to \infty$.
\end{Lemma}

\begin{proof}
For any non-negative function $g$ on $V_N \times \bbZ$ and any finite sets $A, B$, the deterministic bound
	\[ \Big| \sum_{(x,y) \in A} g(x,y) - \sum_{(x,y) \in B} g(x,y) \Big|
	\;\leq\; \sum_{(x,y) \in A \triangle B} g(x,y) \]
holds. Applied to $g = |\psi|^2$, $A = A_+(T)$, and $B = R_{T/N}\setminus R_0$, and using that $A(T) \supseteq R_0$ so that $A_+(T) \triangle (R_{T/N}\setminus R_0) = A(T) \triangle R_{T/N}$, this yields
	\[ |Q_N - W_N| \;\leq\; \frac{1}{N a_N}\sum_{(x,y) \in A(T)\triangle R_{T/N}} |\psi(x,y)|^2\,. \]
On the good event $E$, the integer-valued vertical coordinate satisfies $R_{\lfloor T/N - \Delta_N\rfloor} \subseteq A(T) \subseteq R_{T/N + \Delta_N}$, so
	\[ A(T)\triangle R_{T/N} \;\subseteq\; R_{T/N + \Delta_N} \setminus R_{\lfloor T/N - \Delta_N\rfloor}\,. \]
Hence, by~\eqref{eqn:psi squared},
\begin{align}
|Q_N- W_N|\,  \mathbf{1}_{E}
& \leq \frac{1}{N a_N} \sum_{x \in V_N} \sum_{y=\lceil T/N-\Delta_N\rceil}^{\lfloor T/N + \Delta_N\rfloor} | \psi(x,y)|^2 \nonumber \\
& = \frac{1}{a_N} \sum_{y=\lceil T/N-\Delta_N\rceil}^{\lfloor T/N + \Delta_N\rfloor} \sum_{k=2}^K \alpha_k \!\big( \tfrac{T}{Na_N} \big)^2
\cdot e^{2 q_k^N ( \frac{y}{a_N} - \frac{T}{N a_N} ) } \nonumber \\
& \leq \frac{2 \Delta_N+1}{a_N} \sum_{k=2}^K \alpha_k \!\big( \tfrac{T}{Na_N} \big)^2
\cdot e^{2 q_k^N  \Delta_N / a_N } . \nonumber
\end{align}
The conclusion now follows because of the follow observations:
\begin{itemize}
\item The functions $\a_k$ are continuous so $\a_k( \tfrac{T}{Na_N} ) \to \a_k(y_0)$. This in  particular implies $\a_k( \tfrac{T}{Na_N} ) = O(1)$. 

\item $q_k^N \to \g_k$ is bounded and $\tfrac{\Delta_N}{a_N} \to 0$ so $e^{2 q_k^N \Delta_N/a_N} \to 1$. 

\item Also, the prefactor $\tfrac{2\Delta_N + 1}{a_N} \to 0$.
\end{itemize}
\end{proof}

Finally, we compute the limit of $W_N$.
\begin{Lemma}
\label{clm:W to sigma}
We have $W_N \to \sigma^2(\phi)$ as $N \to \infty$.
\end{Lemma}
\begin{proof}
Since $T = N \lfloor y_0 a_N\rfloor$, we have $h \coloneqq T/N = \lfloor y_0 a_N\rfloor \in \bbZ$. Using~\eqref{eqn:psi squared} and the change of variable $r \coloneqq h - y$,
\begin{align}
W_N & = \frac{1}{a_N} \sum_{y=1}^{h} \sum_{k=2}^K \alpha_k \!\big( \tfrac{T}{N a_N} \big)^2 e^{2 q_k^N ( y - h) /a_N  }
= \frac{1}{a_N} \sum_{k=2}^K \alpha_k \!\big( \tfrac{T}{N a_N} \big)^2   \sum_{r=0}^{h - 1} e^{-2 q_k^N r / a_N }
\nonumber \\
& = \sum_{k=2}^K \alpha_k\!\big( \tfrac{T}{N a_N} \big)^2 \cdot \frac{ 1 - e^{-2 q_k^N T/(N a_N) } }{ a_N( 1 - e^{-2 q_k^N / a_N} ) }\,.
\label{eqn:W to sigma}
\end{align}
To pass to the limit we use three ingredients: 
\begin{itemize}
\item $\tfrac{T}{Na_N} \to y_0$, so $\alpha_k(\tfrac{T}{Na_N})^2 \to \alpha_k(y_0)^2$ by continuity. 
\item $q_k^N \to \gamma_k$, so $e^{-2 q_k^N T/(Na_N)} \to e^{-2 \gamma_k y_0}$. 
\item $\tfrac{q_k^N}{a_N} \to 0$ together with $\tfrac{1 - e^{-u}}{u} \to 1$ as $u \to 0$ give 
$$ a_N(1 - e^{-2 q_k^N/a_N}) = 2 q_k^N \cdot \tfrac{1 - e^{-2 q_k^N/a_N}}{2 \tfrac{q_k^N}{a_N} } \to 2 \gamma_k . $$ 
\end{itemize}
Summing the finitely many $k$-terms,
\[
  W_N \;\to\; \sum_{k=2}^K \alpha_k(y_0)^2\,\frac{1 - e^{-2 \gamma_k y_0}}{2 \gamma_k} \;=\; \sigma^2(\phi)\,.
\]
\end{proof}

We now deduce Theorem \ref{th:GFF} from the above lemmas.

\begin{proof}[Proof of Theorem \ref{th:GFF}]
Recall the martingale differences $\cX_{k,N} \coloneqq (Na_N)^{-1/2}\, \psi(X_k,Y_k)$ for $1 \leq k \leq T$, the quadratic variation $Q_N$, and its deterministic proxy $W_N$. In the proof below we establish the $L^1$ convergence $Q_N \to \sigma^2(\phi)$ by combining Lemmas~\ref{clm:Q on Ec},~\ref{clm:Q-W} and~\ref{clm:W to sigma}; the argument does not require $\sigma^2(\phi) > 0$.

If $\sigma^2(\phi) = 0$, the $L^1$ convergence $Q_N \to 0$ gives $\bbE[M_N(T)^2] = \bbE[Q_N] \to 0$, so $M_N(T) \to 0$ in $L^2$ and hence in probability. By Slutsky's theorem applied to the preamble above, $(Na_N)^{-1/2}\laa D_T, \varphi\raa \to 0$ in probability, which is the claimed conclusion $\Rightarrow \cN(0, 0) = \d_0$.

In the remainder of the proof we assume $\sigma^2(\phi) > 0$ and verify conditions~(a),~(b),~(c) of Theorem~\ref{thm:martingale-CLT}.

Since $Q_N, W_N \geq 0$ and $W_N$ is deterministic, combining Lemmas~\ref{clm:Q on Ec}, \ref{clm:W to sigma}, and~\ref{clm:Q-W} with $\bbP(E^c) \leq N^{-\nu}$ gives
\begin{align*}
\bbE\big( |Q_N - \sigma^2(\phi) | \big)
&\leq \bbE \big( |Q_N - W_N| \mathbf{1}_{E} \big) + \bbE\big( Q_N \mathbf{1}_{E^c} \big) \\
&\qquad + W_N \bbP(E^c) + |W_N - \sigma^2(\phi) |
\;\longrightarrow\; 0
\end{align*}
as $N \to \infty$.
Hence $Q_N \to \sigma^2(\phi)$ in $L^1$, and consequently in probability, thus verifying (c).

To see that (a) holds, note that $\max_{k \leq T} \cX_{k, N}^2 \leq Q_N$ pathwise, so
$$ \bbE \Big( \max_{k \leq T} \cX_{k,N}^2 \Big) \leq \bbE(Q_N)\,. $$
Since we have just proved $Q_N \to \sigma^2(\phi)$ in $L^1$ (above the verification of~(c)), we have $\bbE(Q_N) \to \sigma^2(\phi)$. Hence there exists $N_0$ such that $\bbE(Q_N) \leq \sigma^2(\phi) + 1$ for every $N \geq N_0$. The finitely many remaining $N < N_0$ contribute a finite bound, so $\sup_N \bbE(\max_{k \leq T} \cX_{k, N}^2) < \infty$, i.e.\ condition~(a) of Theorem~\ref{thm:martingale-CLT} is satisfied.

Finally, to prove~(b), recall from~\eqref{eq:truepsi} that
$$ \psi(x, y) = \sum_{j=2}^K \a_j\!\big(\tfrac{T}{Na_N}\big)\, e^{q_j^N(y - \tfrac{T}{N} ) / a_N }  \, f_j^N(x)\,. $$
Cauchy--Schwarz mode-by-mode gives
$$ |\psi(x, y)|^2 \;\leq\; \Big(\sum_{j=2}^K \a_j\!\big(\tfrac{T}{Na_N}\big)^2 \, e^{2 q_j^N(y-\tfrac{T}{N} ) / a_N }\Big) \Big(\sum_{j=2}^K (f_j^N(x))^2\Big)\,. $$
Since $(f_j^N)_{j=1}^N$ is an orthonormal basis of $\bbR^{V_N}$ in the normalized inner product~\eqref{eq:product}, Parseval's identity gives $\sum_{j=1}^N (f_j^N(x))^2 = N$ for every $x \in V_N$, hence $\sum_{j=2}^K (f_j^N(x))^2 \leq N$. Therefore
$$ \cX_{k,N}^2 \;\leq\; \frac{1}{a_N} \sum_{j=2}^K \a_j\!\big(\tfrac{T}{Na_N}\big)^2 \, e^{2 q_j^N (Y_k - \tfrac{T}{N} ) / a_N }  \,. $$
The event $E$ implies that for every $k \leq T$ we have $Y_k \leq T/N + \Delta_N$. So we conclude that
\begin{align*}
\max_{k \leq T} \cX_{k,N}^2 \mathbf{1}_{E} & \leq \sum_{j=2}^K \alpha_j\!\big( \tfrac{T}{N a_N} \big)^2
\cdot \frac{1}{a_N} \exp \Big( \frac{2 q_j^N \Delta_N}{a_N}  \Big) \to 0
\end{align*}
as $N \to \infty$.
Therefore,
\begin{align*}
\bbE \big( \max_{k \leq T} \cX_{k,N}^2 \big) & \leq \bbE \big( \max_{k \leq T} \cX_{k,N}^2 \mathbf{1}_{E} \big) 
+ \bbE ( Q_N \mathbf{1}_{E^c} ) \to 0 ,
\end{align*}
by Lemma \ref{clm:Q on Ec}.  That is, $\max_{k \leq T} |\cX_{k,N}|$ converges to $0$ in $L^2$, and thus in probability as well, verifying (b).

This shows that $\cX_{k, N}$ satisfy the conditions of Theorem \ref{thm:martingale-CLT}, which completes 
the proof of Theorem \ref{th:GFF}.
\end{proof}

\medskip 

\section{Maximal fluctuations bound}\label{sec:max_fluct}
In this section we prove Theorem~\ref{th:max_fluct}. Fix $\n > 0$ and $m \geq 1$, and let $C = C(\n, m) > 0$ be a sufficiently large constant, to be determined below. Throughout, write
\[
  L_T \coloneqq \log N + \log_+(T/N)\,,\qquad
  T^\sharp \coloneqq N\sqrt{\t_{mix}}\,(\log N)^2\,.
\]
With $\Delta_N \coloneqq C \max\!\big\{ \log N,\; \sqrt{((T \wedge T^\sharp)/N)\, L_T} \big\}$, the goal is to show that
\[
  \bbP\big( R_{\lfloor T/N - \Delta_N\rfloor} \subseteq A(T) \subseteq R_{T/N + \Delta_N}\big) \geq 1 - N^{-\n}
\]
for every integer $T$ with $1 \leq T \leq N^m$ and every $N$ large enough. The argument splits according to the size of $T$ into three regimes.

When $T \leq (C/e^3) N \log N$, we have $\Delta_N \asymp \log N$. The inner inclusion is deterministic for $C$ large enough: $\lfloor T/N - \Delta_N\rfloor \leq 0$ and $R_{\lfloor T/N - \Delta_N\rfloor} \subseteq R_0 \subseteq A(T)$. The outer inclusion follows from Proposition~\ref{pr:apriori}-(a) applied with time-coefficient $C/e^3$ and height parameter $\b = C$: the hypothesis $\b > e^2 (C/e^3) = C/e$ holds trivially, and $A(T) \subseteq R_{C \log N} \subseteq R_{T/N + \Delta_N}$ with probability at least $1 - N^{-(C-1)} \geq 1 - N^{-\n}$ for $C \geq \n + 1$.

When $(C/e^3) N \log N \leq T \leq T^\sharp$, we have $T \wedge T^\sharp = T$ and $\Delta_N \asymp \sqrt{(T/N) L_T}$. Propositions~\ref{pr:inner} and~\ref{pr:outer} (Sections~\ref{sec:inner} and~\ref{sec:outer}) give the inner and outer inclusions respectively, each with probability at least $1 - N^{-(\n+1)}$ for $C$ large enough (so that the threshold $(C/e^3) N\log N$ exceeds the $\nu$-dependent lower bounds required by those propositions); a union bound yields the $1 - N^{-\n}$ bound.

When $T^\sharp \leq T \leq N^m$, Proposition~\ref{pr:reduction} (Section~\ref{sec:reduction}) couples $A(T)$ with an auxiliary IDLA process $A_*$ on $\bbG_N$ starting from $R_0$, at a reduced time $T' \leq T^\sharp$ and an integer shift $k = (T - T')/N$, with $A(T) = A_*(T') + (0, k)$ holding with probability at least $1 - N^{-(\n+1)}$. Depending on whether $T' \leq (C/e^3) N\log N$ or $T' \geq (C/e^3) N\log N$, the previous small- or medium-time regime applied at parameter $\n + 1$ to $A_*$ at time $T'$ gives an analogous inclusion for $A_*(T')$ with width $\Delta_N' \coloneqq C \max\{\log N, \sqrt{(T'/N) L_{T'}}\} \leq \Delta_N$, using $T' \leq T^\sharp$ and $L_{T'} \leq L_T$. Shifting by $(0, k)$ and using $T'/N + k = T/N$ gives the claim on the intersection of the two events.

\subsection{Maximal fluctuations inner bound} \label{sec:inner}
Let $(A(t))_{t\leq T}$ be an IDLA process on $\bbG_N$ starting from $A(0) = R_0$, and assume $T \geq C N\log N$ with $C = C(\n, m)$ as in the section opening.
In this section we prove that
	\[ A(T) \supseteq R_{\lfloor T/N - \ell_*\rfloor} \]
with high probability for $N$ large enough, where $\ell_*$ is as in Proposition~\ref{pr:inner}.
In fact, as is standard in IDLA, we prove something stronger, namely that the IDLA cluster is, with high probability, completely filled up to height $\lfloor T/N - \ell_*\rfloor$ \emph{even if the particles are stopped when reaching this height}. Let us make this statement precise.

For $h \geq 1$ integer, denote by $(\bar A_h(t))_{t \geq 0}$ the IDLA process, but where the particles are halted when they reach height $h$. That is, $\bar A_h(0) = R_0$, and for $t \geq 0$ we start a cylinder walk (as defined in~\eqref{def:SRW}) uniformly on
$\{ (x,0) : x \in V_N \}$, and we let $(\bar X_{t+1}, \bar Y_{t+1})$ denote the
exit location of the walk
from  $\bar A_h(t) \cap R_{h-1}$.  Define $\bar A_h(t+1) = \bar A_h(t) \cup \{ (\bar X_{t+1}, \bar Y_{t+1}) \}$. Note that necessarily $\bar Y_{t+1} \leq h$. 
It is immediate to see that, as long as $\bar A_h(t) \subseteq R_{h-1}$, the distribution of $(\bar A_h(s))_{s \leq t}$
is exactly that of the original IDLA process $(A(s))_{s \leq t}$. Moreover, it is ensured by the definition
that $\bar A_h(t) \subseteq R_h$ for all $t\geq 0$.
Thus it follows from a standard application of the abelian property that one may couple $(A(t))_{t\geq 0}$ and $(\bar A_h(t))_{t \geq 0}$ in such a way that $\bar A_h(t) \subseteq A(t)$ for all $t \geq 0$. More generally, for all positive integers $h_1 \leq h_2$ there exists a coupling of $(\bar A_{h_1}(t))_{t\geq 0}$ and $(\bar A_{h_2}(t))_{t\geq 0}$ such that
\begin{equation}\label{eq:monotone_stop}
	\bar A_{h_1}(t) \subseteq \bar A_{h_2}(t) \qquad \text{for every } t \geq 0\,.
\end{equation}

\begin{Proposition}\label{pr:inner}
For every $\n > 0$ there exist constants $C_1 = C_1(\n) < \infty$ and $K_0 = K_0(\n) < \infty$ such that, for every integer $T \geq K_0 N\log N$ and with
	\[ \ell_* \coloneqq C_1 \sqrt{\frac{T}{N}\Big( \log N + \log_+\!\big(T/N\big)\Big)}\,, \qquad h \coloneqq \Big\lfloor \frac{T}{N} - \ell_* \Big\rfloor\,, \]
the IDLA process $(A(t))_{t \leq T}$ on $V_N \times \bbZ$ starting from $A(0) = R_0$ satisfies $h \geq 1$ and
	\[ \bbP \big( A(T) \not\supseteq R_h \big) \leq \bbP\!\big( \bar A_h(T) \not\supseteq R_h \big) \leq N^{-(\n+1)} \]
for $N$ large enough.
\end{Proposition}

\begin{Remark}[Inner-tail family]\label{rk:inner-tail}
The proof below establishes the stronger tail 
	\[ \bbP\!\big( A(T) \not\supseteq R_{\lfloor T/N - \ell \rfloor} \big) \;\leq\; T\, \exp\!\Big(\! -\frac{\ell^2 N}{4T}\Big)\,, \]
valid for every $0 \leq \ell \leq T/(2N)$ with $\lfloor T/N - \ell\rfloor \geq 1$. The proposition is the specialization $\ell = \ell_*$, which makes the right-hand side polynomially small in $N$.
\end{Remark}

The proof adapts the one-step iterative argument of Jerison, Levine and Sheffield~\cite{jerison2012logarithmic} to the general kernel setting, using Freedman's martingale inequality~\cite{freedman1975tail}.

We now return to our argument. With $h = \lfloor T/N - \ell_*\rfloor$, recall that $R_0 \subseteq \bar A_h(T)$ deterministically (the bottom half-cylinder is the initial configuration). The event $\{\bar A_h(T) \not\supseteq R_h\}$ therefore depends only on the finite set $R_h \setminus R_0 = V_N \times \{1, \ldots, h\}$, which has $Nh$ vertices. By the union bound,
	\begin{equation}\label{UB}
	 \bbP \big( \bar A_h(T) \not\supseteq R_h \big) \leq \sum_{\zeta \in R_h \setminus R_0} \bbP\!\big( \zeta \notin \bar A_h(T) \big)\,.
	\end{equation}
We therefore wish to bound $\bbP(\zeta \notin \bar A_h(T))$ for an arbitrary $\zeta \in R_h \setminus R_0$.

Fix such a $\zeta$ and write $\zeta = (\zeta_1, \zeta_2)$ with $\zeta_1 \in V_N$ and $1 \leq \zeta_2 \leq h$.
By ~\eqref{eq:monotone_stop} applied to the integer stopping heights $\zeta_2 \leq h$, there is a coupling with $\bar A_{\zeta_2}(T) \subseteq \bar A_h(T)$. In particular, $\bbP(\zeta \notin \bar A_h(T)) \leq \bbP(\zeta \notin \bar A_{\zeta_2}(T))$: it suffices to prove the bound for the stopped process $\bar A_{\zeta_2}(T)$ with stopping height equal to $\zeta_2$ itself. Renaming $\zeta_2$ to $h$ in the remainder of this argument (so that $\zeta$ now lies on the stopping layer), we may therefore assume without loss of generality that $\zeta_2 = h$.

For every $(x,y) \in V_N \times \bbZ$ define $H_\zeta(x,y)$ to be the probability that the cylinder walk on $V_N \times \bbZ$ started at $(x,y)$ reaches the layer $\{ (x', \zeta_2) : x' \in V_N \}$ for the first time at the vertex $\zeta$, with the convention that $H_\zeta(x,y) = 0$ for $y > \zeta_2$.

Recall that $(\bar X_t, \bar Y_t)_{t \leq T}$ denote the exit locations added by the stopped process to form $\bar A_h(T)$.
Since $H_\zeta$ is discrete harmonic on the half-cylinder $R_{\zeta_2 - 1}$, we obtain that if 
$$ M_\zeta(t) = \sum_{s=1}^t \Big( H_\zeta(\bar X_s,\bar Y_s) - \frac1N \Big) , $$
then $(M_\zeta(t) )_{t\geq 0} $ defines a martingale.  

\begin{Lemma} \label{lem:M lower tail}
For every $0 \leq \ell \leq T/(2N)$,
$$ \bbP( M_\zeta(T) \leq -\ell ) \leq \exp\!\Big( -\frac{\ell^2 N}{4T} \Big)\,. $$
\end{Lemma}

\begin{proof}
Write $M$ in place of $M_\zeta$ for brevity, and set $\x_t \coloneqq H_\zeta(\bar X_{t}, \bar Y_{t}) - 1/N$ for $1 \leq t \leq T$. The $t$-th particle starts uniformly on $V_N \times \{0\}$, and $H_\zeta$ is discrete harmonic on $R_{h-1}$, so optional stopping at the walk's settlement gives
	\[ \bbE\!\big[ H_\zeta(\bar X_t, \bar Y_t) \mid \bar A_h(t-1) \big] = \frac{1}{N} \sum_{x \in V_N} H_\zeta(x, 0) = \frac{1}{N}\,, \]
where the last equality uses that a cylinder walk started uniformly on $V_N \times \{0\}$ first hits level $h$ uniformly on $V_N \times \{h\}$ (as in the proof of Proposition~\ref{pr:apriori}).
Hence $\bbE(\x_t \mid \cF_{t-1}) = 0$ and $\x_t \in [-1/N, 1]$ almost surely. Expanding the square,
	\[ \bbE\!\big[ \x_t^2 \mid \bar A_h(t-1) \big] = \bbE\!\big[ H_\zeta^2 \mid \bar A_h(t-1) \big] - \frac{2}{N}\,\bbE\!\big[ H_\zeta \mid \bar A_h(t-1) \big] + \frac{1}{N^2}
	= \bbE\!\big[ H_\zeta^2 \mid \bar A_h(t-1) \big] - \frac{1}{N^2}\,. \]
Since $H_\zeta \in [0, 1]$ pointwise, we have $H_\zeta^2 \leq H_\zeta$, hence $\bbE[H_\zeta^2 \mid \bar A_h(t-1)] \leq 1/N$, and so
	\[ \bbE\!\big[ \x_t^2 \mid \bar A_h(t-1) \big] \leq \frac{1}{N} - \frac{1}{N^2} \leq \frac{1}{N}\,. \]
The predictable quadratic variation of $M(T) = \sum_{t=1}^T \x_t$ therefore satisfies $V(T) \leq T/N$.
Applying Freedman's inequality (see (1.5) in~\cite{freedman1975tail}) to the martingale $-M$, whose increments $-\x_t \in [-1, 1/N]$ are bounded above by $1$: for every $\l > 0$,
	\[ \bbE \Big[ \exp \big( -\l M(T) - (e^\l-1-\l) V(T) \big) \Big] \leq 1\,. \]
Since $V(T) \leq T/N$ pathwise and $e^\l - 1 - \l \geq 0$, replacing $V(T)$ by $T/N$ in the exponent only makes the exponential larger, so $\bbE[\exp(-\l M(T))] \leq \exp\!\big((e^\l - 1 - \l)\,T/N\big)$. By Markov's inequality applied to $\exp(-\l M(T))$, for every $\l > 0$,
	\[ \bbP(M(T) \leq -\ell) \leq \exp\!\big(-\l \ell + (e^\l - 1 - \l)\,T/N\big)\,. \]
Choose $\l = \ell N/(2T)$, so $\l \leq 1/4$ since $\ell \leq T/(2N)$. For $\l \in [0, 1/2]$, the elementary inequality $e^\l - 1 - \l \leq \l^2$ gives
	\[ \bbP(M(T) \leq -\ell) \leq \exp\!\Big(-\frac{\ell^2 N}{2T} + \frac{\ell^2 N}{4T}\Big) = \exp\!\Big(-\frac{\ell^2 N}{4T}\Big)\,. \qedhere \]
\end{proof}

We can now prove the inner bound of Proposition \ref{pr:inner}.

\begin{proof}[Proof of Proposition \ref{pr:inner}]
Fix $\n > 0$. Set $L_T \coloneqq \log N + \log_+(T/N)$ for brevity, and note that $L_T \geq \log N$. We choose
\[
  C_1 \coloneqq 2\sqrt{\n+4}\,,\qquad K_0 \coloneqq 16\, C_1^2 = 64\,(\n+4)\,,
\]
and check the two conditions (i) $C_1^2 \geq 4(\n + 4)$, so that $\ell_*^2 \cdot N/(4T) = (C_1^2/4)\, L_T \geq (\n + 4)\, L_T$, and (ii) $\ell_* \leq T/(2N)$, for $T \geq K_0 N \log N$ and $N$ large enough. Condition~(i) is immediate by the choice of $C_1$. For condition~(ii), squaring shows that $\ell_* \leq T/(2N)$ is equivalent to $T/N \geq 4 C_1^2\, L_T$. When $T/N \leq N$, we have $\log_+(T/N) \leq \log N$, so $L_T \leq 2\log N$, and our hypothesis $T \geq K_0 N \log N = 16 C_1^2 N \log N$ yields $T/N \geq 16 C_1^2 \log N \geq 8 C_1^2 L_T$, stronger than required. When $T/N > N$, set $u \coloneqq T/N > N$, so $\log_+(T/N) = \log u \geq \log N$ and $L_T \leq 2 \log u$; we need $u \geq 8 C_1^2 \log u$, which holds uniformly for $u \geq N$ once $N$ is large enough (depending on $C_1$, hence on $\n$) since $u / \log u$ is increasing on $(e, \infty)$ and $N/\log N \to \infty$. In either range, $\ell_* \leq T/(2N)$ for $N$ large. Hence $h = \lfloor T/N - \ell_*\rfloor \geq T/(2N) - 1 \geq 1$ for $N$ large.

By the union bound~\eqref{UB} and the reduction following it, it suffices to show that for every $\zeta = (z, h) \in V_N \times \{h\}$,
	\begin{equation}\label{goal_inner}
	\bbP( \zeta \notin \bar A_h(T) ) \leq e^{-(\n + 4)\,L_T}\,.
	\end{equation}
Since $|R_h\setminus R_0| = Nh \leq T$ and, using $e^{L_T} = N e^{\log_+(T/N)} = \max\{N, T\} \geq T$, the union bound combined with~\eqref{goal_inner} gives a total at most $e^{L_T - (\n+4) L_T} = e^{-(\n+3) L_T} \leq N^{-(\n+3)} \leq N^{-(\n+1)}$ for $N$ large (using $L_T \geq \log N$).

It remains to prove~\eqref{goal_inner}. By condition~(i) and Lemma~\ref{lem:M lower tail} applied with $\ell = \ell_*$ (the hypothesis $\ell_* \leq T/(2N)$ holds by condition~(ii)),
$$ \bbP( M(T) \leq -\ell_* ) \leq \exp\!\Big( -\frac{\ell_*^2 N}{4T} \Big) \leq e^{-(\n + 4) L_T}\,, $$
where $M = M_\zeta$ as above.
By definition, $H_\zeta(x,y) = 0$ for every $x \in V_N$ and every $y > h$, $H_\zeta(x, h) = 0$ for $x \neq z$, and $H_\zeta(z, h) = 1$.
Set
	\[ Z \coloneqq \sum_{t=1}^T \mathbf{1}( (\bar X_t, \bar Y_t) = \zeta )\,, \]
so that $Z$ counts the number of particles frozen at $\zeta$.
Since $\pi_N$ is uniform and $H_\zeta$ is harmonic on $R_{h-1}$, summation over $V_N$ gives $\frac{1}{N}\sum_{x \in V_N} H_\zeta(x, y) = 1/N$ for every $1 \leq y < h$, hence
	\[ \sum_{(x,y) \in R_{h-1} \setminus R_0} H_\zeta(x, y) = h - 1\,. \]
Since $H_\zeta \geq 0$ and $\bar A_h(T) \cap (R_{h-1} \setminus R_0) \subseteq R_{h-1} \setminus R_0$,
\begin{align*}
M(T)
& = -\frac{T}{N} + \sum_{t=1}^T \mathbf{1}_{\{\bar Y_t < h\}} H_\zeta(\bar X_t, \bar Y_t) + \sum_{t=1}^T \mathbf{1}_{\{\bar X_t = z,\, \bar Y_t = h\}}
\\ & = -\frac{T}{N} + \sum_{(x,y) \in \bar A_h(T) \cap (R_{h-1} \setminus R_0)} H_\zeta(x, y) + Z
\\ & \leq -\frac{T}{N} + (h-1) + Z\,.
\end{align*}
Since $h \leq T/N - \ell_*$, we have $T/N - (h - 1) \geq \ell_* + 1 > \ell_*$, hence $M(T) < -\ell_* + Z$. On the event $\{ \zeta \notin \bar A_h(T) \}$, we have $Z = 0$, so
	\[ \bbP\big( \zeta \notin \bar A_h(T) \big) \leq \bbP\big( M(T) < -\ell_* \big) \leq e^{-(\n + 4) L_T}\,, \]
which is~\eqref{goal_inner}.

\end{proof}

\medskip 

\subsection{Maximal fluctuations outer bound} \label{sec:outer}
In this section we deduce the outer bound from the inner bound of Proposition~\ref{pr:inner}, which together with Proposition~\ref{pr:reduction} (Section~\ref{sec:reduction}) completes the proof of Theorem~\ref{th:max_fluct}. Throughout, $T \geq C N\log N$ with $C = C(\n, m)$ as in the section opening, and $\ell_*$ and $h \coloneqq \lfloor T/N - \ell_*\rfloor$ are as in Proposition~\ref{pr:inner}. For $C$ sufficiently large, Proposition~\ref{pr:inner} applies both at parameter $\n$ and at parameter $\n+1$ (the latter is needed inside the proof of Proposition~\ref{pr:outer}).

\begin{Proposition}\label{pr:outer}
For every $\n > 0$ there exist $C' = C'(\n) < \infty$ and $K' = K'(\n) < \infty$ such that, for every integer $T \geq K' N \log N$ and with $r \coloneqq C' \ell_*$,
	\[ \bbP \big( A(T) \nsubseteq R_{T/N + r} \big) \leq N^{-(\n+1)} \]
for $N$ large enough.
\end{Proposition}

The proof uses two lemmas. Lemma~\ref{le:wellspread}, using an observation of~\cite{asselah2025internal}, shows that on the high-probability event $\{\bar A_h(T) \supseteq R_{h-1}\setminus R_0\}$, the particles frozen at level $h$ are well spread out: no site of level $h$ carries more than $C_2 \ell_*$ frozen particles. Lemma~\ref{le:released} then controls the height of an auxiliary IDLA process built by releasing a balanced number of particles from the base.

We set up the probability space supporting all of the random variables that appear below. Let $(W_k)_{k=1}^T$ be independent infinite cylinder walks, with $W_k$ started at a vertex of $V_N \times \{0\}$ chosen according to $\pi_N \otimes \delta_0$. Independently of these, let $(U_z)_{z \in R_{h-1}\setminus R_0}$ be independent infinite cylinder walks indexed by sites of $R_{h-1} \setminus R_0$, with $U_z$ started at $z$.

The stopped process $(\bar A_h(t))_{0 \leq t \leq T}$ is built from the walks $(W_k)$ by the construction in Section~\ref{sec:inner}: particle $k$ follows $W_k$ until either reaching a site of $R_{h-1} \setminus R_0$ outside the current cluster (in which case it settles at that site), or reaching level $h$ (in which case it freezes at that level-$h$ site). Both the settling sites and the freezing sites contribute to $\bar A_h(t+1) = \bar A_h(t) \cup \{(\bar X_{t+1}, \bar Y_{t+1})\}$. Let
\[
  G \coloneqq \{\bar A_h(T) \supseteq R_{h-1} \setminus R_0\}\,.
\]
On $G$, every site of $R_{h-1}\setminus R_0$ is settled by some particle. By Proposition~\ref{pr:inner} (which gives the stronger inclusion $\bar A_h(T) \supseteq R_h$ with high probability), $\bbP(G^c) \leq N^{-(\nu+1)}$.

The auxiliary walks $(U_z)$ play the role of ghosts: by the strong Markov property of the cylinder walk, the post-settlement continuation of any particle that settles at a site $z \in R_{h-1}\setminus R_0$ is, conditionally on the cluster history up to settlement, an independent cylinder walk started at $z$. We may therefore couple the post-settlement continuations to the walks $U_z$ at the settling sites, without changing the joint law of the cluster $\bar A_h(T)$ together with the ghosts. Throughout the rest of this section we work with this coupling, so that the ``ghost from $z$'' (in the sense of Lawler--Bramson--Griffeath~\cite{lawler1992internal}) is the walk $U_z$ whenever $z$ is settled.

\begin{Lemma}[Well-spread frozen particles]\label{le:wellspread}
Let $N_x$ denote the number of particles frozen at $(x, h)$ in $\bar A_h(T)$. For every $\n > 0$ there exists $C_2 = C_2(\n) < \infty$ such that
	\[ \bbP\big( G \cap \{ \exists\, x \in V_N : N_x > C_2 \ell_* \} \big) \leq N^{-(\n+1)} \]
for $N$ large enough.
\end{Lemma}

\begin{proof}
Set $L_T \coloneqq \log N + \log_+(T/N)$ as before. Following~\cite{lawler1992internal}, we decompose the first hits at level $h$ into an alive contribution and a \emph{ghost} contribution.

For each $x \in V_N$, let
	\[ M_x \coloneqq \#\big\{ k \in \{1, \ldots, T\} : \text{the alive walk } W_k \text{ first hits level } h \text{ at } (x, h) \big\}\,. \]
Since the walks $(W_k)_{k\leq T}$ are independent and each is started at a uniform vertex of $V_N \times \{0\}$, by vertex-transitivity the marginal distribution of the first level-$h$ hit of $W_k$ is uniform on $V_N \times \{h\}$ (cf.\ the recurrence in the proof of Proposition~\ref{pr:apriori}). Hence $M_x$ is a sum of $T$ independent indicators of probability $1/N$, with $\bbE(M_x) = T/N$.

For each $x \in V_N$, let
	\[ \widetilde L_x \coloneqq \#\big\{ z \in R_{h-1} \setminus R_0 : \text{the auxiliary walk } U_z \text{ first hits level } h \text{ at } (x,h) \big\}\,. \]
Since the walks $(U_z)$ are independent, $\widetilde L_x$ is a sum of $|R_{h-1}\setminus R_0| = N(h-1)$ independent indicators. We compute its mean. For each fixed $1 \leq y \leq h-1$, let $H_{(x,h)}(\cdot, y)$ denote the probability that a cylinder walk started at $(\cdot, y)$ first hits level $h$ at $(x, h)$. By vertex-transitivity (the same argument as for $M_x$, applied at level $y$ instead of level $0$), $\sum_{z_1 \in V_N} H_{(x,h)}(z_1, y) = 1$. Summing over $y \in \{1, \ldots, h-1\}$ gives $\bbE(\widetilde L_x) = h - 1$.

Following the LBG decomposition, we split $M_x$ according to whether the particle actually reached level $h$ or settled earlier:
	\[ M_x = N_x + L_x\,, \]
where the first term $N_x$ counts the $k$ whose particle was frozen at $(x, h)$ (so the alive walk reached level $h$ for the first time at $(x, h)$ and the particle stopped there), and the second term $L_x$ counts the $k$ whose particle settled at some $z \in R_{h-1}\setminus R_0$ before reaching level $h$ but whose alive walk $W_k$ would have first hit level $h$ at $(x, h)$ had it been allowed to continue past the settling site. Equivalently, $L_x$ counts settled sites $z \in R_{h-1}\setminus R_0$ whose ghost first hits level $h$ at $(x, h)$. Under the coupling fixed in the setup above, the ghost from a settled $z$ is the auxiliary walk $U_z$. Hence $L_x = \#\{z \in R_{h-1}\setminus R_0\,:\, z \in \bar A_h(T) \text{ and } U_z \text{ first hits level } h \text{ at } (x,h)\}$. As was first noted in~\cite{asselah2025internal}, on the good event $G$, every $z \in R_{h-1}\setminus R_0$ is settled, so $L_x = \widetilde L_x$. We obtain the pathwise identity
	\begin{equation}\label{eq:Nx-bound}
	 N_x = M_x - \widetilde L_x \qquad \text{on } G\,.
	\end{equation}

Fix $x \in V_N$ and $C_2 \geq 6$. We bound $\bbP(G \cap \{N_x > C_2 \ell_*\})$ by union bound on the two events
	\[ E_1 \coloneqq \Big\{ M_x \geq \frac{T}{N} + \tfrac{C_2}{3}\,\ell_* \Big\}\,, \qquad
	E_2 \coloneqq \Big\{ \widetilde L_x \leq (h-1) - \tfrac{C_2}{3}\,\ell_* \Big\}\,. \]
Indeed, on $G$ and $\{N_x > C_2 \ell_*\}$, identity~\eqref{eq:Nx-bound} gives $M_x - \widetilde L_x > C_2 \ell_*$. If both $E_1$ and $E_2$ failed, we would have
\[
  M_x - \widetilde L_x < \tfrac{T}{N} + \tfrac{C_2}{3}\,\ell_* - (h-1) + \tfrac{C_2}{3}\,\ell_*
  \leq (\ell_* + 2) + \tfrac{2 C_2}{3}\,\ell_* \leq C_2\, \ell_*\,,
\]
for $C_2 \geq 6$ and $\ell_* \geq 2$ (using $T/N - (h-1) \leq \ell_* + 2$ from $h = \lfloor T/N - \ell_*\rfloor$), a contradiction.

We bound $\bbP(E_1)$ and $\bbP(E_2)$ separately by Bernstein's inequality for sums of independent indicators (we do not need $M_x$ and $\widetilde L_x$ to be jointly independent, as the union bound only requires the marginal tails). By the choice of $C_1$ and $K_0$ in the proof of Proposition~\ref{pr:inner}, $\ell_* \leq T/(2N)$ and $h - 1 \leq T/N$. For $E_1$,
\[
  \bbP(E_1) \leq \exp\!\Big( -\frac{(C_2 \ell_*/3)^2}{2 (T/N) + 2 C_2 \ell_*/9} \Big)
  \leq \exp\!\Big( -\frac{C_2^2\,\ell_*^2}{18\, (T/N) + 2 C_2 \ell_*} \Big)\,.
\]
Using $\ell_* \leq T/(2N)$, the denominator is $\leq (18 + C_2)\,(T/N) \leq 2 C_2\, (T/N)$ for $C_2 \geq 18$, so the exponent is at most $-\frac{C_2}{2}\,\ell_*^2\,(N/T) = -\frac{C_2 C_1^2}{2}\,L_T$, where we used $\ell_*^2 = C_1^2\,(T/N)\, L_T$. The same bound applies to $\bbP(E_2)$ with $h-1$ in place of $T/N$. Hence
\[
  \bbP(E_1) + \bbP(E_2) \leq 2 \exp\!\Big( -\frac{C_2 C_1^2}{2}\,L_T \Big)\,.
\]
Taking the union bound over $x \in V_N$ and using $L_T \geq \log N$,
\[
  \bbP\big( G \cap \{\exists\, x \in V_N : N_x > C_2 \ell_*\} \big)
  \leq 2 N \exp\!\Big( -\tfrac{C_2 C_1^2}{2}\,\log N \Big) \leq N^{-(\n+1)}
\]
for $C_2$ large enough depending on $\n$ and $C_1$.
\end{proof}

For a positive integer $m$, let $\widetilde A^{(m)}$ denote the IDLA cluster obtained by starting from $R_0$ and releasing exactly $m$ particles from each site of $V_N \times \{0\}$, in any order (the resulting cluster is independent of the order by the abelian property).

\begin{Lemma}[Height of a balanced release]\label{le:released}
For every $\n > 0$ there exist constants $C_3 = C_3(\n) < \infty$ and $K_1 = K_1(\n) < \infty$ such that, for every integer $m \geq K_1 \log N$,
	\[ \bbP\big( \widetilde A^{(m)} \nsubseteq R_{C_3 m} \big) \leq N^{-(\n+2)} \]
for $N$ large enough.
\end{Lemma}

\begin{proof}
Let $(A''(t))_{t \geq 0}$ be the auxiliary IDLA process started from $R_0$ in which each particle is released from a site of $V_N \times \{0\}$ chosen uniformly at random. Fix $C_4 \coloneqq 4$ and set $\tau \coloneqq \lceil C_4\, m\, N\rceil$. The number $N_x''(\tau)$ of particles released from a given site $x \in V_N$ in $(A''(t))$ up to time $\tau$ is $\mathrm{Bin}(\tau, 1/N)$ with mean at least $C_4 m$. By the Chernoff bound,
	\[ \bbP\big( N_x''(\tau) < m \big) \leq \exp\!\Big( -\tfrac{(C_4 - 1)^2}{2 C_4}\,m \Big) = \exp\!\Big( -\tfrac{9}{8}\, m \Big)\,. \]
A union bound over $x \in V_N$ gives
	\[ \bbP\!\Big( \min_{x \in V_N} N_x''(\tau) < m \Big) \leq N \exp\!\Big( -\tfrac{9}{8}\,m \Big) \leq \exp\!\big( \log N - \tfrac{9}{8}\, K_1 \log N\big) \leq N^{-(\n+3)} \]
for $K_1 = K_1(\n)$ large enough, since $m \geq K_1 \log N$.

On the complementary event, the cluster $A''(\tau)$ contains $\widetilde A^{(m)}$ by the abelian property (the release multiset of $A''(\tau)$ contains the release multiset that defines $\widetilde A^{(m)}$, and releasing extra particles only enlarges the cluster). Hence
	\[ \bbP\big( \widetilde A^{(m)} \nsubseteq R_{C_3 m} \big) \leq N^{-(\n+3)} + \bbP\big( A''(\tau) \nsubseteq R_{C_3 m} \big)\,. \]

Set $\b \coloneqq e^3 > e^2$, and let $K_{\mathrm{ap}} \coloneqq K_0(\b)$ be the constant from Proposition~\ref{pr:apriori}-(b) (we rename it locally to avoid clash with the section-wide $K_0$ from Proposition~\ref{pr:inner}). Since $\tau \geq C_4\, m\, N \geq C_4\, K_1\, N\log N$, enlarging $K_1$ further (to ensure $C_4 K_1 \geq K_{\mathrm{ap}}$) gives $\tau \geq K_{\mathrm{ap}}\, N\log N$, so Proposition~\ref{pr:apriori}-(b) applies at time $\tau$ with this $\b$:
	\[ \bbP\big( A''(\tau) \nsubseteq R_{e^3\,\tau/N} \big) \leq \exp(-e^3\,\tau/N) \leq \exp(-e^3\, C_4\, m) \leq N^{-(\n+3)} \]
for $K_1$ large enough (using $m \geq K_1 \log N$).

Now set
\[
  C_3 \coloneqq e^3\,(C_4 + 1)\,.
\]
Because $\tau \leq C_4 m N + 1 \leq (C_4 + 1)\,m\,N$ for $m \geq 1$, we have
\[
  \frac{e^3\,\tau}{N} \leq e^3\,(C_4 + 1)\,m = C_3\, m\,,
\]
hence $R_{e^3\,\tau/N} \subseteq R_{C_3 m}$, and so
\[
  \{A''(\tau) \nsubseteq R_{C_3 m}\} \subseteq \{A''(\tau) \nsubseteq R_{e^3\,\tau/N}\}\,.
\]
Combining,
\[
  \bbP\big(\widetilde A^{(m)} \nsubseteq R_{C_3 m}\big) \leq 2\, N^{-(\n+3)} \leq N^{-(\n+2)}
\]
for $N$ large.
\end{proof}

\begin{proof}[Proof of Proposition~\ref{pr:outer}]
Throughout this proof we apply Proposition~\ref{pr:inner} and Lemmas~\ref{le:wellspread} and~\ref{le:released} with parameter $\n + 1$ in place of the section-wide $\n$, so that each of the input lemmas produces a failure event of probability at most $N^{-(\n + 2)}$ for $N$ large. To avoid clashing with the section-wide constants $C_1, \ell_*, h$ tied to $\n$, we denote the resulting constants (for this proof only) by
\[
  \tilde C_1 \coloneqq C_1(\n + 1),\quad
  \tilde K_0 \coloneqq K_0(\n + 1),\quad
  \tilde\ell_* \coloneqq \tilde C_1 \sqrt{(T/N) L_T},\quad
  \tilde h \coloneqq \lfloor T/N - \tilde\ell_*\rfloor,
\]
and $\tilde C_2 \coloneqq C_2(\n + 1), \tilde C_3 \coloneqq C_3(\n + 1), \tilde K_1 \coloneqq K_1(\n + 1)$. For $C$ sufficiently large ($C \geq \tilde K_0$), the hypothesis $T \geq C N\log N$ of this section ensures $T \geq \tilde K_0 N\log N$, so Proposition~\ref{pr:inner} applies at parameter $\n+1$. Define
	\[ \tilde G \coloneqq \{\bar A_{\tilde h}(T) \supseteq R_{\tilde h - 1} \setminus R_0\}\,, \]
the analogue of $G$ at parameter $\nu+1$, and let $\tilde N_x$ denote the analogue of $N_x$ for $\bar A_{\tilde h}(T)$. We work with $\tilde G$, $\tilde\ell_*$, $\tilde h$, $\tilde N_x$ in place of their $\n$-versions throughout.

Set $m \coloneqq \lceil \tilde C_2 \tilde\ell_* \rceil + 1$. Since $T \geq \tilde K_0 N\log N$, we have
$\tilde\ell_* \geq \tilde C_1 \sqrt{\tilde K_0}\,\log N$,
hence $m \geq \tilde C_2 \tilde C_1 \sqrt{\tilde K_0}\,\log N$. By replacing $\tilde C_2$ with $\max(\tilde C_2, \tilde K_1/(\tilde C_1\sqrt{\tilde K_0}))$ if necessary, we may assume $m \geq \tilde K_1 \log N$, so that $m$ satisfies the hypothesis of Lemma~\ref{le:released} at parameter $\n + 1$.

Define the good event
	\[ W \coloneqq \tilde G \cap \{ \tilde N_x \leq \tilde C_2 \tilde\ell_* \text{ for every } x \in V_N \}\,, \]
and note that $\bbP(W^c) \leq 2 N^{-(\n + 2)}$ by Proposition~\ref{pr:inner} and Lemma~\ref{le:wellspread}, both at parameter $\n + 1$.

We enlarge the probability space to support, independently of $(W_k)_{k=1}^T$ and the ghosts $(U_z)$, a family $(V_{x,i})_{x \in V_N,\, i \geq 1}$ of independent cylinder walks with $V_{x,i}$ started at $(x, \tilde h)$. By the strong Markov property of the cylinder walk, for each column $x$ we may couple $V_{x, 1}, \ldots, V_{x, \tilde N_x}$ to the post-freeze tails of the $\tilde N_x$ particles that froze at $(x, \tilde h)$ in $\bar A_{\tilde h}(T)$, so that these coincide almost surely; conditionally on the first stage, the remaining $V_{x, \tilde N_x + 1}, \ldots, V_{x, m}$ remain independent cylinder walks started at $(x, \tilde h)$. This coupling does not alter the joint law of $\bar A_{\tilde h}(T)$ and the tail walks.

To carry out the two-stage abelian decomposition correctly, we use the paused version of the stopped process. Let $P_{\tilde h}(T)$ denote the cluster obtained from the same alive walks $W_1, \ldots, W_T$ but with the modification that, when a particle first reaches level $\tilde h$, it is paused at that site without the site being added to the cluster. Sites are added to $P_{\tilde h}(T)$ only when a particle settles strictly below level $\tilde h$, so $P_{\tilde h}(T) \subseteq R_{\tilde h - 1}$. The multiset of pause locations in $P_{\tilde h}(T)$ coincides with the multiset of freeze locations in $\bar A_{\tilde h}(T)$, since the two constructions agree below level $\tilde h$ and particles are paused/frozen at their first hit of level $\tilde h$; in particular, the multiplicities $\tilde N_x$ are the same. By the abelian property of IDLA~\cite[Section~2]{levine2018long}, the cluster $A(T)$ is obtained pathwise by first building $P_{\tilde h}(T)$ from the alive walks $W_1, \ldots, W_T$, and then, for each column $x \in V_N$, resuming the $\tilde N_x$ paused particles from $(x, \tilde h)$ along their post-pause tails $V_{x, 1}, \ldots, V_{x, \tilde N_x}$ (each settling at the first site outside the current cluster).
Define the auxiliary cluster $B$ to be the result of the same two-stage procedure but releasing all $m$ walks $V_{x, 1}, \ldots, V_{x, m}$ per column instead of only the first $\tilde N_x$. On the event $W$, $\tilde N_x \leq \tilde C_2 \tilde\ell_* \leq m - 1$ for every $x$, so the multiset of walks used to build $A(T)$ is a subset of the multiset used to build $B$. By abelian monotonicity in the released multiset,
\[ A(T) \;\subseteq\; B \qquad \text{pathwise on } W\,. \]

We next compare $B$ to the balanced-release cluster $\widetilde A^{(m)}$ from Lemma~\ref{le:released}. We use the following standard monotonicity of IDLA in the initial cluster, which we state as a lemma.

\begin{Lemma}[Monotonicity in the initial cluster]\label{le:mono-init}
Let $\mathcal{B}_1 \subseteq \mathcal{B}_2 \subsetneq V_N \times \bbZ$ be two proper subsets of the cylinder, and let $\o = (\o_1, \o_2, \ldots)$ be a sequence of infinite cylinder walks with prescribed starting sites. For each $i \in \{1, 2\}$ and $k \geq 0$, let $\Phi_k(\mathcal{B}_i)$ denote the cluster obtained from $\mathcal{B}_i$ after the first $k$ walks have been released and settled according to the usual IDLA rule (the $j$-th walk settles at the first site of $\o_j$ outside the current cluster). Assume that for each $i \in \{1, 2\}$ and each $k \geq 0$, the first exit time of $\o_{k+1}$ from $\Phi_k(\mathcal{B}_i)$ is almost surely finite. Then
\[
  \Phi_k(\mathcal{B}_1) \;\subseteq\; \Phi_k(\mathcal{B}_2) \qquad \text{for every } k \geq 0\,.
\]
\end{Lemma}

\begin{proof}
Induct on $k$. The case $k = 0$ is the assumption $\mathcal{B}_1 \subseteq \mathcal{B}_2$. Suppose the inclusion holds at step $k$. Let $\tau_i$ be the first exit time of $\o_{k+1}$ from $\Phi_k(\mathcal{B}_i)$, which is finite almost surely by the exit hypothesis, and let $p_i \coloneqq \o_{k+1}(\tau_i)$ be the corresponding settling site. Since $\Phi_k(\mathcal{B}_1) \subseteq \Phi_k(\mathcal{B}_2)$, we have $\tau_1 \leq \tau_2$.

If $p_1 \in \Phi_k(\mathcal{B}_2)$, then
\[
  \Phi_{k+1}(\mathcal{B}_1) \;=\; \Phi_k(\mathcal{B}_1) \cup \{p_1\} \;\subseteq\; \Phi_k(\mathcal{B}_2) \;\subseteq\; \Phi_{k+1}(\mathcal{B}_2)\,.
\]
Otherwise, $p_1 \notin \Phi_k(\mathcal{B}_2)$, which means the walk $\o_{k+1}$ is already outside $\Phi_k(\mathcal{B}_2)$ at time $\tau_1$; hence $\tau_2 \leq \tau_1$. Combined with $\tau_1 \leq \tau_2$, this gives $\tau_1 = \tau_2$ and $p_1 = p_2$, so
\[
  \Phi_{k+1}(\mathcal{B}_1) \;=\; \Phi_k(\mathcal{B}_1) \cup \{p_1\} \;\subseteq\; \Phi_k(\mathcal{B}_2) \cup \{p_2\} \;=\; \Phi_{k+1}(\mathcal{B}_2)\,.
\]
This closes the induction.
\end{proof}

Conditional on $P_{\tilde h}(T)$, apply Lemma~\ref{le:mono-init} pathwise with $\mathcal{B}_1 = P_{\tilde h}(T)$ (which satisfies $P_{\tilde h}(T) \subseteq R_{\tilde h - 1} \subseteq R_{\tilde h}$, and in particular is a proper subset of $V_N \times \bbZ$) and $\mathcal{B}_2 = R_{\tilde h}$, driving both processes with the same $Nm$ fresh walks $(V_{x, i})_{x \in V_N,\, 1 \leq i \leq m}$ released from level $\tilde h$. The evolving clusters $\Phi_k(\mathcal{B}_i)$ are contained in the finite-height slab $R_{\tilde h + k}$ for every $k \geq 0$, so each walk $V_{x, i}$ needs only to reach vertical level $\tilde h + k + 1$ to exit the current cluster. Because the vertical coordinate of the cylinder walk is a recurrent lazy symmetric random walk on $\bbZ$, this happens almost surely, and the exit hypothesis of Lemma~\ref{le:mono-init} is satisfied. Let $B'$ denote the resulting cluster starting from $R_{\tilde h}$. Then $B \subseteq B'$ pathwise (unconditionally, since $P_{\tilde h}(T) \subseteq R_{\tilde h}$ holds deterministically). By vertical translation invariance of the cylinder kernel, $B' - (0, \tilde h)$ has the same law as $\widetilde A^{(m)}$, where $\widetilde A^{(m)}$ is the balanced-release cluster from Lemma~\ref{le:released}.

On the event $W$, $A(T) \subseteq B \subseteq B'$ pathwise. Since $B' - (0, \tilde h)$ has the same law as $\widetilde A^{(m)}$ by vertical translation invariance of the cylinder kernel, Lemma~\ref{le:released} applied at parameter $\n+1$ gives
\[
  \bbP\big(B' \nsubseteq R_{\tilde h + \tilde C_3 m}\big) \;=\; \bbP\big(\widetilde A^{(m)} \nsubseteq R_{\tilde C_3 m}\big) \;\leq\; N^{-(\n+2)}\,.
\]
Setting $r \coloneqq \tilde C_3 m \leq \tilde C_3 (\tilde C_2 \tilde\ell_* + 2)$, combining with $\tilde h \leq T/N$ we obtain
\[
  \bbP\big( A(T) \nsubseteq R_{T/N + r} \big) \leq \bbP(W^c) + \bbP\big(B' \nsubseteq R_{\tilde h + \tilde C_3 m}\big) \leq 3 N^{-(\n + 2)} \leq N^{-(\n+1)}
\]
for $N$ large enough. Finally, $r \leq (\tilde C_2 \tilde C_3 + 2 \tilde C_3/\tilde\ell_*)\,\tilde\ell_* \leq (\tilde C_2 \tilde C_3 + 1)\,\tilde\ell_*$ for $\tilde\ell_* \geq 2 \tilde C_3$, which holds for $N$ large. Since $\tilde\ell_* = \tilde C_1 \sqrt{(T/N) L_T} = (\tilde C_1/C_1)\,\ell_*$ by the definition of $\tilde\ell_*$, we obtain
\[
  r \;\leq\; (\tilde C_2\,\tilde C_3 + 1)\,\frac{\tilde C_1}{C_1}\,\ell_*\,.
\]
Therefore Proposition~\ref{pr:outer} holds with the explicit constants
\[
  C' \coloneqq C'(\n) \coloneqq (\tilde C_2\,\tilde C_3 + 1)\,\frac{\tilde C_1}{C_1}\,, \qquad K' \coloneqq \tilde K_0 = K_0(\n + 1)\,,
\]
both depending on $\n$ only, completing the proof.
\end{proof}

\subsection{Reduction to polynomial mixing time}\label{sec:reduction}
The inner and outer bounds of Sections~\ref{sec:inner} and~\ref{sec:outer} give Theorem~\ref{th:max_fluct} in the regime $T \leq T^\sharp$. In this subsection we handle the regime $T^\sharp \leq T \leq N^m$ by reducing it to the former via an abelian layered-release coupling due to Silvestri~\cite{silvestri2020internal}.

\begin{Proposition}\label{pr:reduction}
Fix $\n > 0$ and $m \geq 1$, with $T^\sharp$ as in the statement of Theorem~\ref{th:max_fluct}. For every integer $T$ with $T^\sharp \leq T \leq N^m$ and every $N$ large enough, there exist an integer $T'$ with $1 \leq T' \leq T^\sharp$ and a non-negative integer $k$ with $T' + N k = T$, together with a coupling of $A(T)$ and an IDLA process $(A_*(t))_{t \geq 0}$ on $\bbG_N$ starting from $A_*(0) = R_0$, such that
	\[ \bbP\big( A(T) = A_*(T') + R_k \big) \geq 1 - N^{-(\n+1)}\,, \]
where $A_*(T') + R_k \coloneqq \{(x, y+k) : (x, y) \in A_*(T')\}$.
\end{Proposition}

\begin{proof}
This is Step~1 of the proof of~\cite[Theorem~4]{silvestri2020internal}, which produces  a final reduced time bounded by $\alpha N\sqrt{\t_{mix}}\,(\log N)^2$ for a universal constant $\alpha$ that may be made as small as desired by choosing enough iterations; in particular one can arrange $\alpha \leq 1$, so that the final reduced time satisfies $T' \leq T^\sharp$. The identity $T' + Nk = T$ is forced by cluster-size conservation under the set-equality coupling; equivalently, $k$ is the total effective vertical shift accumulated across the iterations, which equals the cumulative raw layer count minus the buffers of size $b\sqrt{\t_{mix}}\log N$ per iteration.
\end{proof}

\section*{Acknowledgements}
This project was initiated during a visit to Cornell University, and we are grateful to Lionel Levine for many stimulating discussions. A.B.\ was partially supported by NSF grant DMS-2202940. A.Y.\ was partially supported by the Israel Science Foundation, grant no.\ 954/21. V.S.\ acknowledges financial support from INdAM--GNAMPA, Indam  CUP B83C24008390005 and from the FIS~2 project ``Understanding pattern formation in nature via complex analysis'' CUP B53C24009510001. 

\medskip

\bibliography{HLbib2}
\bibliographystyle{alpha}
\vfill 
\end{document}